\documentclass{amsart}
\usepackage{graphicx}
\usepackage{amssymb}
\usepackage{amsfonts}
\usepackage{hyperref}
\usepackage{mathrsfs}
\usepackage[margin=3cm]{geometry}

  [1995/08/01 v0.1 Double stroke roman fonts]

\def\ds@whichfont{dsrom}
\relax

\DeclareMathAlphabet{\mathds}{U}{\ds@whichfont}{m}{n}

\swapnumbers
\newtheorem{theorem}{Theorem}[section]
\newtheorem{lemma}[theorem]{Lemma}
\newtheorem{corollary}[theorem]{Corollary}
\newtheorem{proposition}[theorem]{Proposition}
\theoremstyle{definition}

\newtheorem{remark}[theorem]{Remark}
\newtheorem{example}[theorem]{Example}

\numberwithin{equation}{section}
\theoremstyle{plain}

\numberwithin{equation}{section} 
\numberwithin{figure}{section} 
\theoremstyle{plain}
\theoremstyle{plain}
\theoremstyle{remark}
\newtheorem*{acknowledgement*}{Acknowledgement}

\newcommand{\cB}{{\mathcal B}}
\newcommand{\cC}{{\mathcal C}}
\newcommand{\cD}{{\mathcal D}}
\newcommand{\cE}{{\mathcal E}}
\newcommand{\cF}{{\mathcal F}}
\newcommand{\cG}{{\mathcal G}}
\newcommand{\cH}{{\mathcal H}}

\newcommand{\cJ}{{\mathcal J}}

\newcommand{\cL}{{\mathcal L}}
\newcommand{\cM}{{\mathcal M}}

\newcommand{\cR}{{\mathcal R}}

\newcommand{\cW}{{\mathcal W}}
\newcommand{\cX}{{\mathcal X}}
\newcommand{\cY}{{\mathcal Y}}

\newcommand{\te}{{\theta}}

\newcommand{\Om}{{\Omega}}

\newcommand{\ve}{{\varepsilon}}
\newcommand{\del}{{\delta}}
\newcommand{\Del}{{\Delta}}

\newcommand{\Gam}{{\Gamma}}

\newcommand{\sig}{{\sigma}}
\newcommand{\al}{{\alpha}}
\newcommand{\be}{{\beta}}
\newcommand{\ka}{{\kappa}}
\newcommand{\la}{{\lambda}}

\newcommand{\bbE}{{\mathbb E}}

\newcommand{\bbN}{{\mathbb N}}
\newcommand{\bbP}{{\mathbb P}}
\newcommand{\bbR}{{\mathbb R}}

\newcommand{\bbI}{{\mathbb I}}

\begin{document}
\title[]{Convergence rates in the functional CLT for $\al$-mixing triangular arrays}
 \author{Yeor Hafouta \\
\vskip 0.1cm
Department of Mathematics\\
The Ohio State University}
\email{yeor.hafouta@mail.huji.ac.il, hafuta.1@osu.edu}

\maketitle
\markboth{Y. Hafouta}{Functional CLT rates}
\renewcommand{\theequation}{\arabic{section}.\arabic{equation}}
\pagenumbering{arabic}

\begin{abstract}
We obtain convergence rates (in the Levi-Prokhorove metric) in the functional central limit theorem (CLT) for partial sums $S_n=\sum_{j=1}^{n}\xi_{j,n}$ of  triangular arrays $\{\xi_{1,n},\xi_{2,n},...,\xi_{n,n}\}$ satisfying some mixing and moment conditions (which are not necessarily uniform in $n$).
For certain classes of additive functionals of triangular arrays of contracting Markov chains (in the sense of Dobrushin) we obtain rates which are close to the best rates obtained for independent random variables. 
In addition, we obtain close to optimal rates in the usual CLT and
 a moderate deviations principle and some Rosenthal type inequalities.
We will also  discuss applications to  some classes of local statistics (e.g. covariance estimators), as well as expanding non-stationary dynamical systems, which can be reduced to non-uniformly mixing triangular arrays by an approximation argument.
The main novelty here is that 
our  results are obtained without any assumptions about the growth rate of the variance of $S_n$. The result are obtained using a certain type of block decomposition, which, in a sense, reduces the problem to the case when the variance of $S_n$ is ``not negligible" in comparison with the (new) number of summands.  
\end{abstract}

\section{Introduction}\label{Intro}
\subsection{The functional CLT with rates}
A classical textbook materiel in probability theory is the central limit theorem (CLT) for partial sums $S_n=\sum_{j=1}^{n}\xi_{j,n}$ of zero mean  triangular arrays  $\{\xi_{j,n}: 1\leq j\leq n\}$ of independent random variables. One of the first extensions of the CLT to  triangular arrays of non-independent variables dates back to the famous result of Dobrushin \cite{Dub}, who provided optimal conditions for the CLT for additive functionals $\xi_{j,n}=g_{j,n}(\zeta_{j,n})$ of triangular arrays $\{\zeta_{j,n}\}$ of  contracting Markov chains. We refer to \cite{VarSeth} for a more modern presentation of Dobrushin's  CLT and to \cite{Pel} for the CLT under weaker contraction assumptions.
Since then, the CLT was studied under several assumptions for  ``weakly-dependent" (aka ``mixing") triangular arrays $\{\xi_{1,n},...,\xi_{n,n}\}$ (see, for instance, \cite{Volny, Utev1991, Rio1995, Peligrad1996, Ri, RioBook, Neu, NonStatCLT, PelBook}), but without convergence rates in the general case  when the variance $V_n$ of $S_n$ does not grow linearly fast\footnote{The variance exhibits linear growth for partial sums of a (weakly) stationary sequence $\xi_{j,n}=\xi_j$.} in $n$. CLT's for several classes of time-dependent (non-stationary) expanding or hyperbolic dynamical systems were also obtained, see, for instance \cite{Kifer1998, Conze, NTV, HK}.\footnote{We note that there are many results in the stationary case, however we prefer not elaborate a lot about on that  in our presentation since the results in this paper do not require any type of stationarity assumptions.}
 
A stronger version of the CLT is the, so called, functional CLT. Let $S_{k,n}=\sum_{j=1}^k\xi_{j,n}$ and for every $t\in[0,1]$ set
$$
v_n(t)=\min\{1\leq k\leq n:\,\text{Var}(S_{k,n})\geq t V_n\},\, t\in[0,1].
$$ 
Let us consider the random function $W_n(\cdot)$ on $[0,1]$ given by 
$$
W_n(t)=\sig_n^{-1}\sum_{j=1}^{v_n(t)}\xi_{j,n},\,\,\,\sig_n=\sqrt{V_n}=\sqrt{\text{Var}(S_n)}.
$$
Then the functional CLT states that the random function $W_n(\cdot)$ converges in distributed towards the restriction of a standard Brownian motion to $[0,1]$. When $\xi_{j,n}=\xi_j$ forms a weakly-stationary sequence $\{\xi_j\}$ then the functional CLT has been obtained under appropriate mixing and moment conditions by many authors (see, for instance, \cite{PS, Mclish75, Mclish}), where in this case we can replace $v_n(t)$ by $[nt]$. Recently, the functional CLT was obtained under essentially optimal uniform $\rho$-mixing rates in \cite[Theorem 4.1]{MPU} for triangular arrays (and, in particular, for non-stationary sequences).
One of  the main conditions imposed in \cite{MPU} is the following  assumption: there exists a constant $C>0$ so that for all $n\in\bbN$ we have
\begin{equation}\label{Cond1}
\sum_{j=1}^n\text{Var}(\xi_{j,n})\leq C\text{Var}\left(\sum_{j=1}^n\xi_{j,n}\right).
\end{equation}
The results in \cite{MPU} relied on a new martingale approximation technique which, later  on, in \cite{MP} was applied for sufficiently fast uniformly $\al$-mixing triangular arrays, where the main corollary of the general results was a functional CLT under the assumptions that 
\begin{equation}\label{Cond2}
\sum_{j=1}^n\|\xi_{j,n}\|_{2+\del}^2\leq C\text{Var}\left(\sum_{j=1}^n\xi_{j,n}\right)
\end{equation}
for some $\del>0$. In \cite{HafFCLT} we have shown that for both uniform $\rho$ and $\al$-mixing arrays, the functional CLT holds under slightly stronger mixing rates (in comparison with \cite{MPU,MP}) and some additional growth assumptions, but without the additional ``structural" assumptions \eqref{Cond1} or \eqref{Cond2}. In this paper we will obtain convergence rates in the functional CLT for $\al$-mixing triangular arrays, without condition \eqref{Cond2}. For instance, for uniformly bounded arrays so 
that $\sum_{j=1}^{\infty}(\al(j))^{1/p-1/q}<\infty$ for some $q>p>2$ and $\phi(m_0)<\frac12$ for some $m_0\in\bbN$, where $\al(\cdot)$ and $\phi(\cdot)$ are the uniform $\al$ and $\phi$ mixing coefficients of the array (see Section \ref{MixSec}) we show that
$$
d_P(W_n,B)=O(\sig_n^{-w(p)}),\,\, \text{where }w(p)\to 1/4 \,\text{ as }\, p\to\infty 
$$
where $d_P$ is the Levi-Prokhorov metric  and $B$ is a standard Brownian motion. In particular, if, in addition,  $\sum_j(\al(j))^{\del}<\infty$ for all $\del>0$  then
$$
d_P(W_n,B)=O(\sig_n^{-1/4+\epsilon})
$$
for all $\epsilon>0$. Moreover, as will be described in the following paragraphs, our results will be applicable also for non-uniformly mixing\footnote{An array is non-uniformly mixing when for each $k$ the $k$-mixing coefficient ($\al_n(k)$ or $\phi_n(k)$, see \eqref{al def} and \eqref{phi def}, etc)  of the finite sequence $\{\xi_{j,n}:1\leq j\leq n$\} diverges to $\infty$ as $n\to\infty$. This include case  when for some $k_n=o(n)$ the pairs $\xi_{j,n}$ and $\xi_{j+k,n}$ become weakly dependent only when $k>k_n$ (i.e. after $k_n$ steps).}  triangular arrays, for which even the functional CLT itself (i.e. without any rates) seems to be a new result. As will be described below, our results in the non-uniform mixing case include applications to non-uniformly contracting Markov chains and some types of local statistics.

 Even though we have not found any result in the literature about the convergence rate in the functional CLT for general stationary sufficiently fast mixing sequences, it seems to us that approximating by martingales and then using known results for martingales (e.g. \cite[Theorem 4]{Courbot}) yields rates which are at best of order $O(n^{-1/8+\epsilon})=O(\sig_n^{-1/4+\epsilon/2})$ for all $\epsilon>0$, and the described above result provides such rates in the non-stationary case, where the variance of $S_n$ can grow arbitrary slow (still, see Remark \ref{R beta 1}).

The best rate obtained in general\footnote{In the iid case, it seems that the best rate obtained is $O(n^{-1/4}\ln n)=O(\sig_n^{-1/2}\ln \sig_n)$, see \cite{Dud} and  also \cite{Bor} for a survey about rates in the functional CLT. Let us note that the reason that the power $-1/4$ is the best one obtained (in general) is due to the best known approximation rates in the Skorokhod embedding theorem.} for arrays of independent random variables is $\sig_n^{-1/2}\ln \sig_n$,
 which is significantly better than $O(\sig_n^{-1/4+\epsilon})$. For additive functionals of Markov chains we will be able to get close to the $O(\sig_n^{-1/2}\ln \sig_n)$  rates, and one of the consequences of our general results is that for  uniformly bounded functionals $\xi_{j,n}=g_{j,n}(\zeta_{j,n})$ of a triangular arrays of  Markov chain  $\{\zeta_{j,n}: 1\leq j\leq n\}$ with the above  mixing conditions we have
\begin{equation}\label{ForM}
d_P(W_n,B)=O(\sig_n^{-w(p)}),\,\, \text{where }w(p)\to 1/2 \,\text{ as }\, p\to\infty. 
\end{equation}
Thus, if $\sum_j(\al(j))^{\del}<\infty$ for all $\del>0$ and the array is uniformly bounded then
$$
d_P(W_n,B)=O(\sig_n^{-1/2+\epsilon})
$$
for all $\epsilon>0$. In particular,  the rates $O(\sig_n^{-1/2+\epsilon})$ hold true for uniformly contracting Markov chains in the sense of Dobrushin (see Section \ref{SecDob}), as well as for certain type of non-uniformly contacting chains. In both results described above the function $w(p)$ has an explicit form.

Our more general results include rates for triangular arrays which are not uniformly bounded (not even in $L^q$ for some $q$)\footnote{Instead we will have certain growth rates on $K_{q,n}=\max_j\|\xi_{j,n}\|_q$ for some $q>2$.} and are not uniformly mixing. As an application,
 in Section \ref{SecDob}  we  obtain rates for non-uniformly contracting Markov chains $\{\zeta_{j,n}\}$ and several classes of functionals $\xi_{j,n}=g_{j,n}(\zeta_{j,n})$ which are not uniformly bounded in $L^q$. In this setup we will also discuss possible growth rates for $\sig_n^2$ and nonlinear (in $t$ and $n$) behavior of $v_n(t)$ (see Section \ref{SecVarDob}), which will demonstrate the need in considering $v_n(t)$ instead of $[nt]$ in the non-stationary setup. 

We will also we obtain  functional CLT rates  for arrays of ``Markov chains" with conditional memory\footnote{Namely, $\bbP(\zeta_{j,n}\in A|\zeta_{1,n},....,\zeta_{j-1,n})=\bbP(\zeta_{j,n}\in A|\zeta_{j-m_n,n},....,\zeta_{j-1,n})$ for all $j\geq m_n$.} $m_n$ satisfying certain growth conditions  in $n$, and  the rates $O(\sig_n^{-1/2+\ve})$ are achieved, for instance, when $m_n$ grows slower than $\sig_n^\del$ for all $\del>0$. We refer to \cite{FZ} for examples (e.g. covariance estimators) in which $\sig_n^2$ grows linearly fast in $n$, a case in which the allowed growth rate of the memory $m_n$ will be $o(n^\del)$ for all $\del>0$  (we will also get some rates when $m_n=o(n^{\del_0})$ for some $\del_0$ small enough).

As mentioned above, in the iid case, it seems that the best rate obtained is $O(n^{-1/4}\ln n)=O(\sig_n^{-1/2}\ln \sig_n)$.
  Recently, in \cite{AM}  rates of order $O(n^{-1/4+\epsilon})=O(\sig_n^{-1/2+2\ve})$ for all $\epsilon>0$ were obtained (in particular) for certain classes of stationary dynamical systems admitting a tower extension (in the sense of Young) with exponential tails\footnote{When the tails are $O(n^{-p-1})$ the rates obtained in \cite{AM} has the same form as in \eqref{ForM}.}.  In a certain sense Young towers (and other expanding or hyperbolic dynamical systems) are dynamical counterparts of Markov chains, and so our rates in the Markovian case is consistent with \cite{AM}. While the tower structure is more complicated than Markovian assumptions, our results hold true in the nonstationary setup, where $\sig_n^2$  can grow arbitrary slow, and not only linearly fast (as opposed to the stationary case).


\subsection{Additional limit theorems: Berry-Esseen type estimates, moment type estimates and moderate deviations principle}

Another result obtained in this paper is almost optimal convergence rate in the CLT for $\hat S_n=(S_n-\bbE[S_n])/\sig_n$, where $S_n=\sum_{j=1}^{n}\xi_{j,n}$, $\sig_n=\sqrt{\text{Var}S_n}$ and $\{\xi_{1,n},...,\xi_{j,n}\}$ is a triangular array satisfying certain mixing and moment conditions, which are not necessarily uniform in $n$.  
In \cite{DMR} a rate of order $\sig_{n}^{-1/2}\sqrt{\ln\sig_n}$ was obtained for $\rho$-mixing sequences (which also holds for uniformly mixing triangular arrays), which seems to be the first result of this kind in the general non-stationary setup, where the variance of $S_n$ is not assumed to grow at a certain rate.
 Here, for sufficiently fast mixing arrays  we improve this rate.
We show (in particular) that if $\{\xi_{1,j},...,\xi_{n,n}\}$ is exponentially fast $\alpha$-mixing (uniformly in $n$) and uniformly bounded then
the following almost optimal convergence rate in the CLT holds,
\begin{equation}\label{Rate}
\sup_{t\in\bbR}\left|\bbP(\hat S_n\leq t)-\Phi(t)\right|=O(\sig_n^{-1}\ln^3\sig_n)
\end{equation}
assuming also that $\phi(n_0)<\frac12$ for some $n_0$. In particular \eqref{Rate} holds true when $\phi(n)$ decays exponentially fast as $n\to\infty$.
When the uniform $\alpha$ mixing coefficients  decays as $O(j^{-\te})$ when $j\to\infty$ for some $\te>1$ we obtain the rate $\sig_n^{-\eta(\te)}$, where $\eta(\te)$ converges converges monotonically  to $1$ as $\te\to\infty$ (and has an explicit form).  We also obtain rates when the mixing coefficients decay (possibly polynomially fast) non-uniformly in $n$. For instance, when, roughly speaking, the ``amount of non-uniformity" in $n$ is $O(\ln \sig_n)$, and we have exponentially fast $\alpha$-mixing, then we still get rates of the form \eqref{Rate} but with $\ln^{4}\sig_n$ instead of 
$\ln^3\sig_n$. Other rates are obtained when the amount of non-uniformity is of magnitude $O(\sig_n^\del)$ for $\del$ small enough.
All of these results are obtained using  the so called Stein-Tikhomirov method together with a ``reduction" to the case when the variance of $S_n$ grows  sufficiently fast in the number of summands $n$. In Section \ref{SecDob} we will provide several examples of possible rates in the case of non-uniformly mixing triangular arrays (and  summands $\xi_{j,n}$ which are not even uniformly bounded in some $L^q$ for some $q$).

The rates in \eqref{Rate} are close to the optimal rate $\sig_n^{-1}$, and the main novelty in \eqref{Rate} is that they are obtained under no assumptions on the growth rates of $\text{Var}(S_n)=\sig_n^2$. When the latter grows linearly fast in $n$, then the optimal rate $\sig_n^{-1}$ was obtained for several classes of ``mixing" non-stationary sequences, see \cite{St1, Sku, Tik, Ri, HK, HafNonlin, DavH VV, HafYT}. In \cite{DolgH} optimal rates were obtained for bounded additive functionals of uniformly elliptic inhomogeneous Markov chains, without any growth rates on $\sig_n$ other than $\sig_n\to\infty$, and \eqref{Rate} shows that without the Markovian assumption, using only mixing properties we get almost optimal rates.

Finally, we recall that a moderate deviations principle deals  with asymptotic
behavior of probabilities of the form $\bbP(S_n/c_n\in\Gamma)$, where $c_n$ is a certain normalizing sequence with respect to a certain speed function $s_n$. When the $\phi$-mixing coefficients of the array decay stretched exponentially fast (possibly not uniformly in $n$) we prove, in particular, that there are $\gamma_1,\gamma_2>0$ so that for every set Borel set $\Gamma\subset\bbR$ whose closure coincides\footnote{e.g an interval} with the closure of its interior we have 
$$
\lim_{n\to\infty}\sig_n^{-\gamma_1}\bbP(\sig_n^{-1-\gamma_2}S_n\in\Gamma)=-\frac12\inf_{x\in\Gamma}x^2.
$$
In the terminology of \cite{DemZet}, such results can be regarded as a large deviations principle with the rate function $I(x)=-x^2/2$. 
For sufficiently fast mixing sequences, when the variance of $S_n$ grows  linearly fast, 
such results were obtained using the, so-called, method of cumulants \cite{SaulStat} (see also \cite{Douk2,Ha} and references therein for additional applications of this method for mixing sequences whose variance grows linearly fast in $n$).
 This method has proven to be effective also  as for locally dependent random variables \cite{Dor}. 
The moderate deviations in this paper are also obtained using the method of cumulants, combined with a block partition argument.
 As a by product,  we are also able to obtain a certain type of Rosenthal type inequalities, showing that for each integer $p>2$,
$$
\bbE[(S_n-\bbE[S_n])^p]=\sig_n^p\bbE[Z^p]+O(\sig_n^{p-1})
$$ 
where $Z$ is a standard normal random variable. In particular,
$$
\lim_{n\to\infty}\bbE[(\hat S_n)^p]=\bbE[Z^p], \,\,\hat S_n=(S_n-\bbE[S_n])/\sig_n.
$$

\section{Preliminaries and main abstract results}\label{Main}
In this section we will describe our main results in an abstract form. The readers who are interested to see concrete examples at this stage are referred to Sections \ref{SecDob} and \ref{Local} (especially to Section \ref{SecDob} where all the results are formulated explicitly and independently).

\subsection{Non-uniform mixing arrays}\label{MixSec}
In this section we describe the type of triangular arrays considered in  this paper, present the corresponding arrays of mixing coefficients and our main assumptions concerning these arrays.

Let $\{\zeta_{1,n},\zeta_{2,n},...,\zeta_{n,n}\}$ be a triangular array of random variables defined on some probability space $(\Om,\cF,\bbP)$, taking values on measurable spaces $\cX_{0,n},\cX_{1,n},\cX_{2,n},...,\cX_{n,n}$, respectively. 
Let  $g_{j,n}:\cX_{j,n}\to\bbR$ be measurable functions and set $\xi_{j,n}=g_{j,n}(\zeta_{j,n})$.
We assume here that $\xi_{j,n}$ are square integrable.
In this paper we  obtain various quantitative limit theorems for random variables related to the partial sums
$$
S_{k}^{(n)}=\sum_{j=1}^k\xi_{j,n},\,k\leq n
$$
under several moments and mixing conditions, which are not necessarily uniform in $n$.
We will assume here that $\bbE[\xi_{j,n}]=0$, which is not really a restriction since we can always replace $g_{j,n}$ with $g_{j,n}-\bbE[g_{j,n}(\zeta_{j,n})]$.
Set $S_n=S_{n}^{(n)}$, $V_n=\text{Var}(S_n)$ and $\sig_n=\sqrt{\text{Var}(S_n)}$. In this manuscript we will always assume that $\lim_{n\to\infty}V_n=\infty$, but, expect for  some specific applications, we will not assume any growth rates.
Let us also assume that for some $q>2$ we have 
$$
K_{q,n}=\max_{j\leq n}\|\xi_{j,n}\|_{q}<\infty.
$$
Our result will make use of the following two types of mixing (weak-dependence) coefficients.
 Recall that for any two sub-$\sig$-algebras $\cG$ and $\cH$ of $\cF$ the $\al$-dependence (mixing) coefficient of $\cG$ and $\cH$ is given by
\[
\al(\cG,\cH)=\sup\{|\bbP(A\cap B)-\bbP(A)\bbP(B)|: A\in\cG, B\in\cH\}.
\] 
 Then the $k$-th $\al$-mixing coefficient of the finite sequence $\{\zeta_{j,n}:\,1\leq j\leq n\}$ is defined by 
\begin{equation}\label{al def}
\al_n(k)=\sup\{\al(\cF_n(s),\cF_n(s+k,n)): s\leq n-k\},\, k\leq n
\end{equation}
where $\cF_n(s)$ is the $\sig$-algebra generated by $\zeta_{j,n}, j\leq s$ and $\cF_n(s+k,n)$ is the $\sig$-algebra generated by $\zeta_{j,n}, s+k\leq j\leq n$. The uniform $\al$-mixing coefficients of the array $\zeta$ are given by $\al(j)=\sup_n\al_n(j)$.

Some of our result will be effective under conditions which involve the so-called $\phi$-mixing coefficients, defined as follows.
For any two sub-$\sig$-algebras $\cG$ and $\cH$ of $\cF$ we set
$$
\phi(\cG,\cH)=\sup\{|\bbP(B|A)-\bbP(B)|: A\in\cG, B\in\cH, \bbP(A)>0\}.
$$
The $k$-th $\phi$-mixing coefficient of the a finite sequence $\{\zeta_{j,n}:\,1\leq j\leq n\}$ is defined by
\begin{equation}\label{phi def}
\phi_n(k)=\sup\{\phi(\cF_n(s),\cF_n(s+k,n)): s\leq n-k\},\, k\leq n.
\end{equation}
The uniform $\phi$-mixing coefficients of the entire array are given by $\phi(k)=\sup_n\phi_n(k)$.


\subsubsection{\textbf{Arrays of mixing coefficients and sequences of non-uniform growth}}\label{SecMixCoef}

Let us now set 
$$
S_{k,m}^{(n)}=\sum_{j=k}^{k+m-1}\xi_{j,n},\,\, S_n=S_{1,n}^{(n)}
$$ 
and let $V_n=\bbE[S_n^2]=\text{Var}(S_n)=\sig_n^2$.

\subsubsection{A quantification of the amount of non-uniformity}

Our main results involve the following sequences which, in a sense, measure the ``amount of nonuniformity" of the array.

 
Let us take some $2<p<q$  and set 
\begin{equation}\label{Gamma n def}
\Gamma_n(x)=\sum_{x\leq j\leq n}(\al_n(j))^{1/p-1/q}.
\end{equation}
Let $r_n=\max\left(\Gamma_n^{-1}(\frac1{8 K_{q,n}}),1\right)$ and
\begin{equation}\label{A_n def}
A_n=18r_n(1+K_{2,n})
\end{equation}
where $K_{r,n}=\max_j\|\xi_{j,n}\|_{r}$ for every $r$ and  $\Gamma_n^{-1}(y)$ is the smallest positive integer $m$ so that $\Gamma_n(m)\leq y$.
 We assume here that 
\begin{equation}\label{AA}
A_n=o(\sig_n)
\end{equation}
which is a certain restriction on the ``amount of nonuniformity" of mixing and moment conditions that the arrays is ``allowed" to have.

Next, for each $n$ let us denote by $\be_n=\beta_n(q,A_n)$ the smallest number greater or equal to $1$ so that  
for all $n$ and $k,m$ such that 
$\max_{0<l\leq m}\|S_{k,l}^{(n)}\|_2\leq 4A_n$ we have
\begin{equation}\label{MaxMom}
\left\|\max\{|S_{k,l}^{(n)}|:  0<l \leq m\}\right\|_{q}\leq \beta_n A_n.
\end{equation}
Let us note that $\beta_n$ is well defined and finite since it is the maximum of all possible ratios between the expressions on the left hand side of \eqref{MaxMom} and $A_n$  when  $m+k\leq n$ (the latter expressions are finite since $K_{q,n}<\infty$).

\begin{remark}\label{R beta}
We only need to use $\beta_n$ defined through \eqref{MaxMom} when the variance $\text{Var}(S_n)=\sig_n^2$ does not satisfy $c_1 n\leq \sig_n^2\leq c_2 n$ for some $c_1,c_2>0$ and all $n$ large enough (the exact meaning of this will be clear after we present the main results, see Remark \ref{R beta 1}). 
 While in general we do not assume any growth rate  on $\sig_n^2$, in some statistical applications (e.g. \cite{FZ}) we have $\sig_n^2/n\to\sig^2>0$ for some non-uniformly mixing arrays. In Section \ref{Local} we will discuss applications to a class of examples similar to \cite{FZ} (but under weaker assumptions like $\sig_n^2\geq cn^{\gamma}$ for some $\gamma>0$). 
\end{remark}

\subsection{Upper bounds on $\beta_n$ from \eqref{MaxMom}: a discussion}\label{MMsec}
While the formulation of our main results does not require any assumptions on $\beta_n$, the main results in this paper will only be effective  when $\beta_n=o(\sig_n^{t_0})$ for some $t_0$ small enough (depending on the circumstances) since some power of $\beta_n$ will always appear as a multiplicative constant in the upper bounds in our quantitative results. 
Before formulating our main results let us explain how we can obtain effective upper bounds on $\beta_n$.

Let us assume that 
\begin{equation}\label{jn ass}
\phi_n(j_n)<\frac12-\ve\, \text{ for some }\,j_n<n\, \text{ and a constant }\,\ve>0.
\end{equation}
 Then by \cite[Theorem 6.17]{PelBook}\footnote{It is clear from the proof that, in the notations of  \cite[Theorem 6.17]{PelBook}, when $\phi(n_0)<\frac12-\ve$ the constant appearing there depends on $\phi(n_0)$ only through $\ve$.}, there is a constant $C_{\ve}$ so that for every $q>2$  and all relevant $k,m$,
\begin{equation}\label{Pel}
\left\|\max\{|S_{k,l}^{(n)}|:  0\leq l\leq m\}\right\|_{q}\leq C_{\ve}q\left(j_n\left\|\max_{k\leq j<k+m}|\xi_{j,n}|\right\|_{q}+\max_{0<l\leq m}\|S_{k,l}^{(n)}\|_2\right).
\end{equation}
Hence we have
$$\beta_n\leq C_\ve q\left(1+j_n\left\|\max_{1\leq j\leq n}|\xi_{j,n}|\right\|_{q}\right).$$ 
Let us provide  more specific upper bounds on $\beta_n$ in two cases.

\begin{enumerate}
\item First, it is clear that
$$\beta_n\leq C_\ve q\left(1+j_nK_{\infty,n}\right),\,\,K_{\infty,n}=\max_j\|\xi_{j,n}\|_\infty.$$
As noted above, our results will only be effective when $\beta_n=O(\sig_n^{t_0})$ for $t_0>0$ small enough, depending on the circumstances. Hence, we can always work under the assumption that $$j_nK_{\infty,n}=O(\sig_n^{t_0}).$$
In particular, our results will be effective when $\sup_n\max_j\|\xi_{j,n}\|_\infty<\infty$ (i.e. for uniformly bounded arrays) and $\sup_{n} j_n<\infty$. We refer to Section \ref{SecDob} for examples of non-uniformly mixing triangular arrays (contracting Markov chains) and  non-uniformly bounded functions $g_{j,n}$ for which the condition  $j_nK_{\infty,n}=O(\sig_n^{t_0})$ is satisfied with an arbitrary $t_0$. 

\vskip0.2cm
\item Suppose now that $\sig_n\geq Cn^{\del}$ for some $\del>0$. Let $q_0\geq q$. Since
$$
\max_{k\leq j<k+m}|\xi_{j,n}|^{q_0}\leq\sum_{j=1}^n|\xi_{j,n}|^{q_0}
$$
we have
$$\beta_n\leq C_\ve q\left(1+j_n n^{1/q_0}K_{q_0,n}\right)$$ where $K_{q_0,n}=\max_j\|\xi_{j,n}\|_{q_0}$.
Hence if $j_n K_{q_0,n}=O(n^a)$ then 
$$
\beta_n=O(n^{a+\frac1{q_0}})
$$
which will yield effective results when $a$ is small enough and $q_0$ is large enough. In Section \ref{SecDob} we will provide several examples\footnote{See also an example in Section \ref{Sec7.2}.} in which $\sig_n\geq Cn^{\del}$ and $\beta_n=O(n^{a+\frac1q_0})$ for arbitrary $\del,a$ and $q_0$. We would also like to refer to \cite{FZ} for examples in which $\sig_n^2$ grows linearly fast in $n$.
\end{enumerate}



\subsection{A functional CLT with rates}
 For each $t\in[0,1]$, set $v_n(t)=\min\{1\leq k\leq n: \sig_{k,n}^2\geq t\sig_n^2\}$ where $\sig_{k,n}^2$ is the variance of $\sum_{j=1}^k\xi_{j,n}$. Consider   the random function 
$$
W_n(t)=\sig_n^{-1}\sum_{j=1}^{v_n(t)}\xi_{j,n}=\sig_n^{-1}S_{v_n(t)}^{(n)}
$$
on $[0,1]$. Then $W_n(\cdot)$ is a random element of the  Skorokhod space $D[0,1]$. Let us consider $D[0,1]$ as a metric space with the uniform metric $d(f,g)=\sup_{t\in[0,1]}|f(t)-g(t)|$. Recall that the Prokhorov (or the Levi-Prokhorov) distance between two probability distributions $\mu,\nu$ on $D[0,1]$ is given by 
$$
d_P(\mu,\nu)=\inf\{\ve>0: \,\mu(B)\leq \nu(B^\ve)+\ve\,\text{ and }\,\nu(B)\leq \mu(B^\ve)+\ve\,\text{ for all Borel sets }B\}
$$
where $B^\ve$ is the $\ve$-neighborhood of $B$. 
When $X,Y$ are $D[0,1]$-valued random variables with laws $\mu_X$ and $\mu_Y$, respectively, we abuse the notation and write $d_P(X,Y):=d_P(\mu_X,\mu_Y)$. Then a sequence of $D[0,1]$-valued random variables $Z_n$ converges in distribution to a $D[0,1]$-valued random variable $Z$ if and only if $d_P(Z_n,Z)$ converges to $0$. 

\begin{remark}
When $\xi_{j,n}=\xi_j$ forms a (weakly) stationary sequence then $v_n(t)\thickapprox nt$, however for nonstationary sequences (and triangular arrays) $v_n(t)$ is not necessarily (essentially) a linear function of $t$ or $n$. This is demonstrated in Section \ref{SecVarDob} where we show that for certain types of functionals of triangular arrays of contracting Markov chains, $v_n(t)$ behaves like $t^{q}\sig_n^2$ for some $0<q<1$ (in this case $\sig_n^2$ grows like $n^{a}$ for some $a\in(0,1)$).  For some other functionals it behaves like $n^{a(t)t}$ for some function $a(t)$ so that $C_1\leq a(t)\leq C_2$, where $C_1$ and $C_2$ are positive constants (in this case $\sig_n^2$ grows logarithmically fast in $n$). 
\end{remark}

Our first result is as follows.

\begin{theorem}\label{FunCLT} 
Assume that $\lim_{n\to\infty}\sig_n=\infty$ and that $A_n=o(\sig_n)$. Let $l_n$ be sequence so that $l_n\leq\frac{\sig_n^2}{18 A_n^2}$, where  $A_n$ was defined in \eqref{A_n def}. Then there is a constant $C_{p}$ which depends only on $p$ so that 
\begin{equation}\label{ProkEst}
d_P(W_n,B)\leq C_{p}C(n;q)\left(\sig_n^{-\frac{p-2}{2p}}|\ln \sig_n|^{3/4}+
\mathfrak q_n^{\frac{p}{2p+4}}|\ln \mathfrak q_n|^{1/2}\right)
\end{equation}
where $\mathfrak q_n=l_n^{1/2}\sig_n^{-1}+l_n\sig_n^{-2(1-2/p)}+l_n^{-1/2}$ and
$$
C(n;q)=(K_{q,n}+1)A_n\beta_n\Gamma_n(1).
$$
\end{theorem}
For uniformly mixing and bounded arrays we get the following result.
\begin{corollary}\label{Corr1}
Suppose that the triangular array $\{\xi_{j,n}:1\leq j\leq n\}$ is uniformly bounded, that $\sum_{j=1}^{\infty}(\al(j))^{1/p-1/q}<\infty$ for some $q>p>2$ and $\phi(m_0)<\frac12$ for some $m_0\in\bbN$, where $\al(\cdot)$ and $\phi(\cdot)$ are the uniform $\al$ and $\phi$ mixing coefficients of the array (see Section \ref{MixSec}). Then 
$$
d_P(W_n,B)=O(\sig_n^{-w_1(p)}),\,\, \text{where }w_1(p)\to 1/4 \,\text{ as }\, p\to\infty. 
$$
 Thus, if $\sum_j(\al(j))^{\del}<\infty$ for all $\del>0$ (e.g. $\al(j)$ decays stretched exponentially fast) and the array is uniformly bounded then
$$
d_P(W_n,B)=O(\sig_n^{-1/4+\epsilon})
$$
for all $\epsilon>0$.
\end{corollary}

Next, as discussed in Section \ref{Intro}, the rates in Corollary \ref{Corr1}  are not as good as the rates $O(\sig_n^{-1/2}\ln \sig_n)$ obtained in the independent case. To get closer to the latter rates, let us introduce the following ``memory" coefficients:
\begin{equation}\label{r def}
r_n(p,m)=\max_{m\leq j\leq n}\left\|\sum_{s=j+1}^{n}\left(\bbE[\xi_{s,j}|\zeta_{1,n},...,\zeta_{j,n}]-\bbE[\xi_{s,j}|\zeta_{j-m,n},...,\zeta_{j,n}]\right)\right\|_{p}.
\end{equation}
 If $\{\zeta_{j,n}: 1\leq k\leq n\}$ is a  Markov chain with memory $m_n$ then $r_n(p,m_n)=0$ for all $p$. In general,
$r_n(p,m)$ measures how close $\{\xi_{j,n}: 1\leq j\leq n\}$ is to be a hidden Markov chain with memory $m$ in the $L^{p}$-norm (see also Remark \ref{Re} below).

\begin{theorem}\label{FuncCLT MC}
Assume that $\lim_{n\to\infty}\sig_n=\infty$ and that $A_n=o(\sig_n)$.
Let $l_n$ be a sequence of positive integers so that $l_n\leq\frac{\sig_n^2}{18 A_n^2}$. Then there is a constant $C_p$ so that
\begin{equation}\label{ProkEst3}
d_P(W_n,B)\leq C_{p}C(n;q)\left(\sig_n^{-\frac{p-2}{2p}}|\ln \sig_n|^{3/4}+
w_n^{\frac{p}{2p+4}}|\ln w_n|^{1/2}\right) 
\end{equation}
where $w_n=l_n\sig_n^{-2(1-2/p)}+l_n^{1/2}\sig_n^{-1}+r_n(p,[l_n/2])l_n^{-1/2}$.
\end{theorem}

As discussed in Section \ref{Intro},
the best rate obtained for independent arrays is $O(\sig_n^{-1/2}\ln \sig_n)$, which is significantly better than $O(\sig_n^{-1/4+\epsilon})$. For Markov chains we will be able to get such better rates, and one of the consequences of our general result is the following:
\begin{corollary}\label{Corr2}
Suppose $\{\zeta_{j,n}:1\leq j\leq n\}$ is a triangular array of Markov chains\footnote{Namely $\bbP(\zeta_{j,n}\in A|\zeta_{1,n},...,\zeta_{j-1,n})=\bbP(\zeta_{j,n}\in A |\zeta_{j-1,n})$ for all $n$ and $j$ and a measurable set $A$.} and let $\xi_{j,n}=g_{j,n}(\zeta_{j,n})$ for some array $g_{j,n}$ of uniformly bounded functions in $L^q$. Then, under the same mixing conditions as in Corollary \ref{Corr1} we have 
$$
d_P(W_n,B)=O(\sig_n^{-w_2(p)}),\,\, \text{where }w_2(p)\to 1/2 \,\text{ as }\, p\to\infty. 
$$
 If $\sum_j(\al(j))^{\del}<\infty$ for all $\del>0$ and the array is uniformly bounded then
$$
d_P(W_n,B)=O(\sig_n^{-1/2+\epsilon})
$$
for all $\epsilon>0$.
\end{corollary}

We refer to Section \ref{SecDob} for an application of Theorem \ref{FuncCLT MC} for non-uniformly mixing Markov chains and summands $\xi_{j,n}$ which are not necessarily uniformly bounded in $L^q$. We also refer to Remark \ref{RR} for an applications of Theorem \ref{FuncCLT MC} for additive functional of non-uniformly mixing triangular arrays of Markov chains with memory $m_n$ for some $m_n=o(\sig_n)$ (which is applicable to the covariance estimators in \cite{FZ} in the Markovian case).
\begin{remark}
In both theorems the function $w_i(p), i=1,2$ has an explicit form.
\end{remark}

\begin{remark}\label{Re}
Let us also consider
$$
c_n(p,m)=\sup_{m\leq j\leq n}\sup_{s>j}\left\|\bbE[\xi_{s,n}|\zeta_{1,n},...,\zeta_{j,n}]-\bbE[\xi_{s,n}|\zeta_{j-m,n},...\zeta_{j-1,n},\zeta_{j,n}]\right\|_{p}.
$$
Then by \eqref{Var po rel},
\begin{equation}\label{cbn.11}
r_n(p,m)\leq mc_n(p,m)+CK_{q,n}\sum_{m<k\leq n}\varpi_{q,p,n}(k)
\leq mc_n(p,m)+2CK_{q,n}\sum_{m<k\leq n}(\al_n(k))^{1/p-1/q}.
\end{equation}
Therefore, Theorem \ref {FuncCLT MC} yields upper bounds on $d(W_n,B)$ in terms of the ``memory" coefficients $c_n(p_0,\cdot)$ which are more local in nature, together with the ''tails" $\sum_{k>m}(\al_n(k))^{1/p-1/q}.$
\end{remark}

\subsection{Additional results}
The following result provides close to optimal rates in the non-functional CLT.

\begin{theorem}\label{BEthm}
Suppose that $\lim_{n\to\infty}\sig_n=\infty$ and that  $q\in(2,3]$.
\vskip0.2cm 
(i)If $\al_n(j)\leq B_n j^{-a_n}$ for some $B_n,a_n>1$ then  
$$
\sup_{t\in\bbR}|\bbP(S_n/\sig_n\leq t)-\Phi(t)|\leq R_n(1+(1-a_n)^{-1/2})\sig_n^{-(1-2\gamma_n)}
$$
where $\Phi$ is the standard normal distribution function,
$$
\gamma_n=\frac{q-1}{a_n(q-2)(2q)^{-1}+2(q-1)}
$$
and
 $R_n=O\left(\beta_n^3B_n^3A_n^3\right)$.
\vskip0.2cm 
(ii) If there exist $\ka_n>0$ and $B_n\geq 1$ so that $\al_n(j)\leq B_ne^{-\ka_n j}$  then 
$$
\sup_{t\in\bbR}|\bbP(S_n/\sig_n\leq t)-\Phi(t)|\leq R_n(1+\ka_n^{-3})\sig_n^{-1}\ln^3(\sig_n).
$$
\end{theorem}

\begin{corollary}\label{Cor BE}
Suppose that $\lim_{n\to\infty}\sig_n=\infty$ and that  $q\in(2,3]$. Moreover, assume that $\sup_n\max_j\|\xi_{j,n}\|_{\infty}<\infty$ and that $\phi(m_0)<\frac12$ for some $m_0\in\bbN$. Then: 

\vskip0.1cm
(i)If $\al_n(j)\leq B j^{-a}$ for some constants $B,a>1$ then  
$$
\sup_{t\in\bbR}|\bbP(S_n/\sig_n\leq t)-\Phi(t)|=O(\sig_n^{-(1-2\gamma)})
$$
where
$
\gamma=\gamma_a=\frac{q-1}{a(s-2)(2s)^{-1}+2(q-1)}.
$

\vskip0.2cm 
(ii) If there exist constants  $\ka>0$ and $B>1$ so that $\al_n(j)\leq e^{-\ka j}$  then 
$$
\sup_{t\in\bbR}|\bbP(S_n/\sig_n\leq t)-\Phi(t)|=O(\sig_n^{-1}\ln^3(\sig_n)).
$$
\end{corollary}
In Sections \ref{SecDob} and \ref{Local} we will apply Theorem \ref{BEthm} for specific examples of non-uniformly mixing triangular arrays for which $R_n$ and either $(1-a_n)^{-1/2}$ or $\ka_n^{-3}$ grow slower than $\sig_n^{a}$ for $a<1/2$.

Our next result concerns large deviations. 
\begin{theorem}\label{MDPthm}
Suppose that $\sig_n\to\infty$, that $\xi_{j,n}$ are bounded and that  $\phi_n(j)\leq B_ne^{-a_nj^\eta}$ for some sequences $a_n>0$, $B_n\geq1$ such that $(a_n)$ is bounded and a constant $\eta>0$. Fix an arbitrary $u>1$ and set
$$
R_n=A_n(K_{\infty,n}j_n+1)B_ne^{a_n^{-u}}
$$
where $K_{\infty,n}=\max_{j\leq n}\|\xi_{j,n}\|_{L^\infty}$ and  $A_n$ was defined in \eqref{A_n def}. In addition, assume that $R_n=o(\sig_n^{1/3})$.
Let $t_n$ be a sequence so that $\lim_{n\to\infty}t_n=\infty$ and $\lim_{n\to\infty}t_n(R_n^{-3}\sig_n)^{-\frac{1}{1+2\gamma}}=0$, $\gamma=\eta^{-1}$. Set
 $W_n=\frac{S_n}{\sig_n t_n}$. Then $W_n$ obeys
the following moderate deviations principle (MDP):
for every  Borel measurable set $\Gam\subset\bbR$,
\[
-\inf_{x\in\Gam^o}I(x)\leq \liminf_{N\to\infty}\frac1{t_n^2}\ln\bbP(W_n\in\Gamma)\leq  \limsup_{n\to\infty}\frac1{t_n^2}\ln\bbP(W_n\in\Gamma) \leq-\inf_{x\in\bar\Gam}I(x)
\]
where $\Gam^o$ denotes the interior of a set $\Gam$, $\bar\Gam$ denotes its closure and $I(x)=x^2/2$.
\end{theorem}

\begin{corollary}
If $\phi(j)\leq Be^{-aj^\eta}$ for some constants $a>0$ and $B\geq1$ and $\sup_n\max_j\|\xi_{j,n}\|_{L^\infty}<\infty$ then  for every sequence $t_n$  so that $\lim_{n\to\infty}t_n=\infty$ and $\lim_{n\to\infty}t_n(\sig_n)^{-\frac{1}{1+2\gamma}}=0$ we have the following:
for every  Borel measurable set $\Gam\subset\bbR$,
\[
-\inf_{x\in\Gam^o}I(x)\leq \liminf_{n\to\infty}\frac1{t_n^2}\ln\bbP(W_n\in\Gamma)\leq  \limsup_{n\to\infty}\frac1{t_n^2}\ln\bbP(W_n\in\Gamma) \leq-\inf_{x\in\bar\Gam}I(x).
\]
\end{corollary}

In Section \ref{SecDob} we will apply Theorem \ref{MDPthm} for specific examples of triangular arrays where $R_n\to\infty$  and $B_n\to\infty$ while $a_n$ might decay to $0$. 

As a by product of the proof of Theorem \ref{MDPthm}, we are also able to prove the following type of Gaussian moment estimates.
\begin{theorem}\label{RoseIneq}
Under the conditions of Theorem \ref{MDPthm} we have the following.
Let $Z$ be a standard normal random variable. Then there is a constant $C>0$ so that for every integer $p>2$ we have
$$
\left|\bbE[S_n^p]-\sig_n^p\bbE[Z^p]\right|\leq R_nC^p(p!)^{2+\eta^{-1}}\sum_{1\leq u\leq (p-1)/2}\frac{R_n^u\sig_n^{2u}p^u}{(u!)^2}=O(R_n^p\sig_n^{p-1}).
$$
Thus, if $R_n^p/\sig_n=o(1)$ we have $\lim_{n\to\infty}\bbE[(S_n/\sig_n)^p]=\bbE[Z^p]$. In particular, the latter convergence holds for all $p$  if $R_n=o(\sig_n^{\ve})$ for all $\ve>0$ (which is the case\footnote{In this case $R_n$ is bounded.} when $\sup_{n}\max_j\|\xi_{j,n}\|_\infty<\infty$ and $\phi(j)\leq Ae^{-aj^\eta}$ for some positive constants $A,a,\eta$).
\end{theorem}
When $R_n^p=o(\sig_n)$ for all $p$ this theorem yields the CLT via the method of moments, and so it can also be viewed as a certain type of quantified version of the CLT (explicit examples when $R_n\to\infty$ at a given rate will be discussed in Section \ref{SecDob}).

\subsection{Additional remarks }

\begin{remark}
In Section \ref{SecDob} we will discuss applications to non-uniformly contracting Markov chains which have non-uniform mixing rates, and for functionals $\xi_{j,n}=g_{j,n}(\zeta_{j,n})$ which are not uniformly bounded in $L^q$ for some $q$.  
\end{remark}

\begin{remark}[Linear growth case]\label{R beta 1}
As mentioned in Remark \ref{R beta}, when the variance $\sig_n^2$ grows linearly fast in $n$ our proofs proceed without using $\beta_n$ defined through \eqref{MaxMom}. This means that in the case of linear growth all the results hold true after (technically) setting $\beta_n=K_{q,n}$ and $A_n=1$. Hence, in that case there is no need in using the $\phi$-mixing coefficients since they are only used in this paper in order to provide effective upper bounds on $\beta_n$.
In particular, when $\xi_{j,n}=\xi_j$ forms a stationary sequence we get the  rates $\sig_n^{-w_i(p)}$ when $\sum_{n}(\al(n))^{1/p-1/q}<\infty$ and $\xi_{j,n}=\xi_j\in L^q$. Let us also note that, for the best of our knowledge, close to optimal functional CLT rates were not obtained previously even in the stationary case.
\end{remark}

\section{Applications to non-uniformly contracting Markov chains}\label{SecDob}
In this section we will analyze our results in a particular classical example: contracting Markov chains (in the sense of  Dobrushin). Since we will be considering non-uniformly contracting Markov chains, the resulting arrays will be non-uniformly mixing, and hence the results in this section also address the question of concrete examples of such arrays which fit the general framework described in Section \ref{Main} (see section \ref{Local} for more examples).

Let us begin with introducing the setup. We us assume that for each $n$ the finite sequence $\{\zeta_{j,n}:\,1\leq j\leq n\}$ forms a Markov chain. Let us denote by $\pi_i^{(n)}(x,A)=\bbP(\zeta_{i+1.n}\in A|\zeta_{i,n}=x)$ the corresponding regular conditional probabilities. Recall that the Dobrushin contraction coefficient of a regular family of conditional probabilities $\pi(x,A)$ is given by 
$$
\del(\pi)=\sup_{x,y}\|\pi(x,\cdot)-\pi(y,\cdot)\|_{TV}.
$$
Let $\del_n=\max\{\del(\pi_i^{(n)}): 1\leq i<n\}$ and $C_n=K_{\infty,n}=\max\{\|\xi_{j,n}\|_{L^\infty}: 1\leq j\leq n\}$. Then the classical Dobrushin's CLT states that  $S_n/\sig_n$ obeys the CLT if $\sig_n\to\infty$ and 
$$
\lim_{n\to\infty}C_n(1-\del_n)^3\sig_n^2=\infty
$$
(which is an optimal condition).

\subsection{On the growth rates of $\sig_n^2$ and $v_n(t)$}\label{SecVarDob}
In this section we discuss possible behaviors of $\sig_n^2$ and $v_n(t)$. 
First, by combining \cite[Proposition 13]{Pel} with \cite[Lemma 4.1]{VarSeth} we see that for each $n$ and all $1\leq \ell<k\leq n$ we have
\begin{equation}\label{Var}
\frac{1-\sqrt{\del_n}}{1+\sqrt{\del_n}}\sum_{j=\ell}^k\text{Var}(\xi_{j,n})\leq \text{Var}\left(\sum_{j=\ell}^k \xi_{j,n}\right)\leq \frac{1+\sqrt{\del_n}}{1-\sqrt{\del_n}}\sum_{j=\ell}^k\text{Var}(\xi_{j,n}).
\end{equation}

In view of \eqref{Var} (applied with $\ell=1$ and $k=n$), by imposing conditions on $\text{Var}(\xi_{j,n})$ we see that $\sig_n^2$ can have an arbitrary growth rate.  Moreover, \eqref{Var} can be used to estimate $v_n(t)$, and to provide examples where $v_n(t)$ is not essentially a linear function of $t$ or $n$, in contrast with the stationary case. For the sake of clarity, some results in this direction are formulated in the following:

\begin{lemma}\label{Lvar}
(i) Suppose that $\sup_n\del_n<1$. 
\vskip0.1cm
(1) If $\text{Var}(\xi_{k,n})\asymp k^{-\gamma}$ for some $\gamma<1$  then   $\sig_n^2\asymp n^{1-\gamma}$ and 
$$
v_n(t)=C^{\pm 1}t^{\frac 1{1-\gamma}}n=C^{\pm 1}(t\sig_n^2)^{\frac 1{1-\gamma}}
$$
where $a=C^{\pm 1}b$ means that $C^{-1}b\leq a\leq Cb$ for some constant $C>0$ which does not depend on $a$ or $b$ and $v_{k,n}\asymp a_k$ means that $v_{k,n}/a_k$ is bounded and bounded away from $0$ uniformly in $k$ and $n$.
\vskip0.1cm
(2)   If $\text{Var}(\xi_{k,n})\asymp k^{-1}$ then   $\sig_n^2\asymp \ln n$ and  
$$
 n^{C_1t}\leq v_n(t)\leq n^{C_2t}
$$
for some constants $C_1,C_2>0$.

\vskip0.2cm
(ii) Let $\ve_n=1-\del_n$ and suppose $\ve_n>0$ with $\ve_n\to0$.  
\vskip0.1cm
(1) If $\text{Var}(\xi_{k,n})\asymp k^{-\gamma}$ for some $0<\gamma<1$ then  
$$
C_2n^{1-\gamma}\ve_n\leq \sig_n^2\leq C_2n^{1-\gamma}(\ve_n)^{-1}
$$
and
$$
 C_1 t^{\frac 2{1-\gamma}}n(\ve_n)^{\frac 2{1-\gamma}}\leq v_n(t)\leq C_2 t^{\frac 1{1-\gamma}}n(\ve_n)^{-\frac 2{1-\gamma}}
$$
where $C_1,C_2$ are positive constants.
\vskip0.1cm
(2)   If $\text{Var}(\xi_{k,n})\asymp k^{-1}$ then   $c_1\ve_n \ln n\leq \sig_n^2\leq c_2\ve_n^{-1}\ln n$ and  
$$
 n^{C_1\ve_n^2 t}\leq v_n(t)\leq n^{C_2\ve_n^{-2}t}
$$
for some constants $c_1,c_2,C_1,C_2>0$.

\end{lemma}
\begin{proof}
Consider first the case when $\sup_n\del_n<1$. Then by \eqref{Var} we have $\sig_{k,n}^2=C^{\pm 1}\sum_{j=1}^k V_k(n)$, where $V_k(n)=\text{Var}(\xi_{k,n})$. Combining this with  the definition of $v_n(t)$ we see that
\begin{equation}\label{FF0}
C_1t\sig_n^2\leq  \sum_{j=1}^{v_n(t)} V_j(n)\leq C_2t\sig_n^2+V_{v_n(t)}(n)
\end{equation}
for some constants $C_1,C_2>0$ which depend only on $\sup_n\del_n$. Now, consider the case when $V_k(n)\asymp k^{-\gamma}$ for $\gamma\in(0,1)$.  Then 
we get from \eqref{FF0} that 
$$
(v_n(t))^{1-\gamma}=C^{\pm 1}t\sig_n^2.
$$
Now,  by \eqref{Var} we have $\sig_n^2\asymp n^{1-\gamma}$  
and so 
$$
v_n(t)=C^{\pm 1}(t\sig_n^2)^{\frac 1{1-\gamma}}=C^{\pm 1}t^{\frac 1{1-\gamma}}n.
$$
When $V_k(n)\asymp k^{-1}$ then we conclude from \eqref{Var} that $\sig_n\asymp \ln n$ and hence we derive from \eqref{FF0} that
$$
C_1 t\ln n \leq \ln v_n(t)\leq C_2t\ln n.
$$

When $\del_n$ is not uniformly bounded away from $1$, we similarly get from \eqref{Var} and the definition of $v_n(t)$ that
$$
C_1t\ve_n\sig_n^2\leq  \sum_{j=1}^{v_n(t)} V_j(n)\leq C_2t\sig_n^2/\ve_n+V_{v_n(t)}(n).
$$
Hence, if  $V_k(n)\asymp k^{-\gamma}$ we get that 
$$
C_1 tn^{1-\gamma}\ve_n^2\leq (v_n(t))^{1-\gamma}\leq C_2 tn^{1-\gamma}/\ve_n^2
$$
and so 
$$
 C_1 t^{\frac 1{1-\gamma}}n(\ve_n)^{\frac 2{1-\gamma}}\leq v_n(t)\leq C_2 t^{\frac 1{1-\gamma}}n(\ve_n)^{-\frac 2{1-\gamma}}.
$$
Similarly, when    $\text{Var}(\xi_{k,n})\asymp k^{-1}$  then by \eqref{Var} we have  $\ve_n \ln n\leq  \sig_n^2\leq \ve_n^{-1}\ln n$ and so
$$
C_1 t\ve_n \ln n\leq  \ln v_n(t)\leq C_2t\ve_n^{-1}\ln n.
$$
\end{proof}

\begin{remark}
The purpose of Lemma \ref{Lvar} is to demonstrate how we can control $\sig_n^2$ and $v_n(t)$, and
similar  results can be formulated when $\text{Var}(\xi_{k,n})\asymp c_n k^{-\gamma}$, $\gamma\leq 1$ for some sequence $c_n$ which grows sufficiently slow in $n$, as well as under different growth rates which, for instance,  insure that $\sum_{j=1}^n\text{Var}(\xi_{k,n})$ is of order $n^{\te}$ or $(\ln n)^\te$ for some $\te$, etc.
\end{remark}

\subsection{General estimates on the $\phi$-mixing coefficients}
\begin{lemma}\label{MixLem}
For all $n$ and all $k<n$ 
we have $\phi_n(k)\leq \del_n^k$.
\end{lemma}
\begin{proof}
Fix some $n$ and  $m,k\in\bbN$ so that $m+k\leq n$.
Let $G=g(\zeta_{m+k,n},\zeta_{m+k+1,n},...)$ be a bounded  nonnegative  function of $\zeta_{j,n}, n\geq j\geq m+k$. Then by the Markov property
$$
\bbE[G|\zeta_{1,n},...,\zeta_{m,n}]=\bbE[\bbE[G|\zeta_{1,n},...,\zeta_{m+k,n}]|\zeta_{1,n},...,\zeta_{m,n}]=\bbE[G_{m+k,n}|\zeta_{m,n}]
$$
where $G_{m+k.n}=\bbE[G|\zeta_{m+k,n}]$. 
Next, write
$$
\bbE[G_{m+k,n}|\zeta_{m,n}]=\int G_{m+k}(y)\pi_m^{(n,k)}(\zeta_{m,n},dy)
$$
where $\pi_m^{(n,k)}(x,A)=\bbP(\zeta_{m+k,n}\in A|\zeta_{m,n}=x)$. Hence, if $\mu_{m,n}$ denotes the law of $\zeta_{m,n}$ then for almost every realization of $\xi_{m,n}$ we have
$$
\left|\bbE[G|\zeta_{1,n},...,\zeta_{m,n}]-\bbE[G]\right|=\left|\int G_{m+k,n}(y)\pi_m^{(n,k)}(\zeta_{m,n},dy)-\int\left(\int G_{m+k,n}(y)\pi_m^{(n,k)}(w,dy)\right)d\mu_{m,n}(w)\right|
$$
$$
=\left| \int\left( \int G_{m+k,n}(y)[\pi_m^{(n,k)}(\zeta_{m,n},dy)-\pi_m^{(n,k)}(w,dy)]\right)d\mu_{m,n}(w)\right|\leq \sup|G|\del(\pi_m^{(n,k)})\leq 
 \sup|G|\del_n^k.
$$
Now, set  $X=(\zeta_{m+k,n},...,\zeta_{n,n})$ and $Y=(\zeta_{1,n},...,\zeta_{m,n})$ and let $A$ be a measurable set on the state space of $Y$ so that $\bbP((\xi_{1,n},...,\xi_{m,n})\in A)>0$, and $B$ be a measurable set on the state space of $X$. Let $g$ be the indicator function  of the set $B$.
Then
$$
\bbP(X\in B, Y\in A)-\bbP(X\in B)\bbP(Y\in A)=\bbE[\bbI(Y\in A)\big(\bbP[X\in B|Y]-\bbP(X\in B)\big)]
$$ 
$$
=\bbE[\bbI(Y\in A)\big(\bbE[G|Y]-\bbE[G]\big)]
$$
and so 
$$
\left|\bbP(X\in B, Y\in A)-\bbP(X\in B)\bbP(Y\in A)\right|\leq\bbE[\bbI(Y\in A)]\del_n^k=\bbP(Y\in A)\del_n^k.
$$
Thus,
$$
\left|\bbP(X\in B|Y\in A)-\bbP(X\in B)\right|\leq \del_n^k
$$
and the proof of the lemma is complete since $\phi_n(k)$ is the smallest upper bound on the left hand side when $A$ and $B$ range over all possible relevant measurable sets and $m$ ranges over all possible choices of indexes so that $k+m\leq n$.
\end{proof}

\subsection{On bounding $A_n$, $\Gamma_n(x)$ and $\beta_n$}

The following result is an immediate consequence of Lemma \ref{MixLem} and the definitions \eqref{Gamma n def} and \eqref{A_n def} of $A_n$ and $\Gamma_n$.
\begin{corollary}\label{C1}
We have
 $\Gamma_n(x)\leq\frac{\del_n^{x(1/p-1/q)}}{1-\del_n}$. Hence,  if $K_{q,n}=o(\ve_n^{-1})$ then  $\Gamma_n^{-1}(\frac1{8 K_{q,n}})\leq h_n:=C|\ln(\ve_n)|/\ve_n$
where $C=C_{p,q}$  is some positive constant. In this case, we have
$$
A_n=O\left(h_n(K_{2,n}+1)\right) 
$$
and the condition $A_n=o(\sig_n)$ holds true when 
\begin{equation}\label{Con}
\sum_{j=1}^n\text{Var}(\xi_{j,n})\gg |\ln(\ve_n)|(1+K_{2,n})\ve_n^{-2}.
\end{equation}
\end{corollary}

The appearance of the term $K_{2,n}\ve_n^{-1}$ on the above left hand side imposes another restriction on the behavior of the variances. For instance, condition \eqref{Con} is in force when $\ve_n\geq cn^{-\te}$ for some $\te,c>0$,  $\text{Var}(\xi_{j,n})\gg n^{-\gamma}$, $\gamma\in(0,1)$ with $2\te<1-\gamma$ and
 $K_{2,n}=o(n^{1-\gamma-\te})$.

\begin{lemma}\label{L0}
Let $0<\ve<\frac12$ and set
$$
j_n= \frac{|\ln(\frac12-\ve)|}{|\ln(\del_n)|}=O(1/\ve_n).
$$
Then $\beta_n$ from \eqref{MaxMom} satisfies 
$$
\beta_n\leq C_\ve q\left(1+j_n\left\|\max_{1\leq j\leq n}|\xi_{j,n}|\right\|_{q}\right). 
$$
Moreover:
$$\beta_n\leq C_\ve q\left(1+j_nK_{\infty,n}\right),\,\,K_{\infty,n}=\max_j\|\xi_{j,n}\|_\infty.$$
Furthermore, if $\sig_n\geq c n^\del$ (which holds  true when $\sum_{j=1}^n\text{Var}(\xi_{j,n})\geq c\ve_n^{-1}n^{2\del}$) for some $\del>0$ then
$$\beta_n\leq C_\ve q\left(1+j_n n^{1/q}K_{q,n}\right).$$
\end{lemma}

\begin{proof}
The lemma follows from the discussion in Section \ref{MMsec} together with the observation that by Lemma \ref{MixLem} we have $\phi_n(j_n)\leq \frac12-\ve$. Note that the condition for the lower bound of the variance is indeed sufficient due to \eqref{Var}.
\end{proof}

\subsection{Functional CLT rates}
In this section we will prove the following result by applying Theorem \ref{FuncCLT MC}.
\begin{corollary}\label{Dob Cor}
(i) If $\limsup\del_n<1$ and $\sup_{n,j}\|\xi_{j,n}\|_{L^\infty}<\infty$ then $d_P(W_n, B)=O(\sig_n^{-1/2+\ve})$ for every $\ve>0$.

(ii) Fix some  $q>p>2$. If $K_{q,n}=o(\ve_n^{-1})$ and  
\begin{equation}\label{Con0}
(1+K_{q,n})\big(\frac{|\ln \ve_n|}{\ve_n}(K_{2,n}+1)\big)(1+\ve_n^{-1}K_{\infty,n})\ve_n^{-1}\leq C(\ve_n \sum_{j=1}^n\text{Var}(\xi_{j,n}))^{r_0} 
\end{equation}
for some $r_0<w(p):=\min(\frac{p-2}{2p}, \frac{p-2}{p+2},\frac{p}{2p-4})$ (note that 
$w(p)\to 1/2$ as $p\to\infty$), then
$$
d_P(W_n, B)\leq C_q\sig_n^{-(w(p)-r_0)}|\ln \sig_n|.
$$

(iii) Suppose that 
 $\ve_n \sum_{j=1}^n\text{Var}(\xi_{j,n})\geq cn^{2\del}$ for some $\del,c>0$ and all $n$ large enough
and 
\begin{equation}\label{Con1}
(1+K_{q,n})\big(\frac{|\ln \ve_n|}{\ve_n}(K_{2,n}+1)\big)(1+n^{1/q}\ve_n^{-1}K_{q,n})\ve_n^{-1}\leq C n^{\del r_0}
\end{equation}
for some $r_0<w_p$.
 Then $\sig_n\geq c_1n^{\del}$ for some constant $c_2>0$ and
$$
d_P(W_n, B)\leq C_q\sig_n^{-(w(p)-r_0)}|\ln \sig_n|.
$$

\end{corollary}

Conditions \eqref{Con0} and \eqref{Con1}  impose  restrictions which only involve $\ve_n=1-\del_n$ and the individual variances of  $\xi_{j,n}$, and below we will give a few sufficient conditions under assumptions on the decay rates of $\ve_n$.
\begin{example}
(i) Condition \eqref{Con0} is in force when $\text{Var}(\xi_{j,n})\asymp j^{-\gamma}$ and $\ve_n\geq cn^{-\te}$ for some positive $\te$, $c$ and $\gamma\in(0,1)$ and $K_{\infty, n}=O(n^{\eta})$ for some $\eta$ so that both $3\eta+3\te$ and $2\eta+3\te$ are smaller than $r_0(1-\gamma-\te)$ (so we can take sufficiently small $\eta$ and $\te$, given $\gamma$ and $r_0$). 
\vskip0.1cm
(ii) Condition \eqref{Con1}  is valid if $\ve_n\geq cn^{-\te}$, $\text{Var}(\xi_{j,n})\geq c j^{2\del+\te-1}$ and $K_{q, n}=O(n^{\eta})$ for some $\eta$ so that both $3\eta+2\te+1/q$ and $2\eta+3\te+1/q$ are smaller than $2r_0\del$ (so we can take sufficiently small $\eta$ and $\te$ and a large $q$, given $\del$ and $r_0$).
\end{example}

\begin{proof}[Proof of Corollary \ref{Dob Cor}]
(i) This part is a direct consequence of Corollary \ref{Corr2}. Indeed, $\al_n(k)\leq\phi_n(k)\leq \del_n^k\leq \del^k$ for some $0<\del<1$ and so for every $a>0$ we have $\sup_n\sum_{k}\al_n(k)^{a}<\infty$. 
\vskip0.1cm
(ii) Let us take $l_n=2$ in  Theorem \ref{FuncCLT MC}. Then $r_n(p,[l_n/2])=0$ and so the number $w_n$ appearing in Theorem \ref{FuncCLT MC} satisfies
$$
w_n=2\sig_n^{-2(1-2/p)}+\sqrt 2\sig_n^{-1}.
$$
Next, by Corollary \ref{C1} and Lemma \ref{L0} we have that 
$$
C(n;q)\leq (1+K_{q,n})\big(\frac{|\ln \ve_n|}{\ve_n}(K_{2,n}+1)\big)(1+\ve_n^{-1}K_{\infty,n})\ve_n^{-1}
$$
where we have used that $\Gamma_n(1)\leq \frac1{1-\del_n^{1/p-1/q}}\leq C_{p,q}\frac1{1-\del_n}=C_{p,q}\ve_n^{-1}$.
Combining this with \eqref{Var} and the conditions in part (ii) we see that 
$$
C(n;q)=O(\sig_n^{r_0}).
$$
Now the result follows from Theorem \ref{FuncCLT MC}.
\vskip0.1cm
(iii) First, the lower bound $\sig_n\geq cn^{\del}$ follows from the assumption of Corollary \ref{Dob Cor} together with \eqref{Var}. Next,
let us take again $l_n=2$ in  Theorem \ref{FuncCLT MC}.
 By Corollary \ref{C1} and Lemma \ref{L0} we have that 
$$
C(n;q)\leq(1+K_{q,n})\big(\frac{|\ln \ve_n|}{\ve_n}(K_{2,n}+1)\big)(1+n^{1/q}\ve_n^{-1}K_{q,n})\ve_n^{-1}.
$$
Combining this with \eqref{Var} and the conditions in part (iii) we see that 
$$
C(n;q)=O(\sig_n^{r_0}).
$$
Now the result follows from Theorem \ref{FuncCLT MC}.

\end{proof}

\subsection{Berry-Essen type estimates}

\subsubsection{Non-uniform polynomial and exponential rates}
To verify the conditions of Theorem \ref{BEthm} (in order to obtain CLT rates) we will need the following result. 

\begin{lemma}\label{LL}
Let $\ve_n=1-\del_n>0$. 
\vskip0.1cm
(i) If $\ve_n$ is bounded away from $0$ then $\phi_n(j)\leq e^{-\ka j}$ for some $\ka>0$.
\vskip0.1cm
(ii) Suppose that $\ve_n\to 0$. Then:
\vskip0.1cm
(1) We have $\phi_n(j)\leq B_nj^{-a_n}$ if $\ln B_n\geq Ca_n\ln (a_n/\ve_n)$ for some constant $C>1$.
\vskip0.1cm
(2) We have $\phi_j(n)\leq B_ne^{-\ka_n j}$ if $|\ln(\del_n)|\geq \ka_n-\ln B_n/n$ . In particular, by taking $B_n=1$ we see that $\phi_j(n)\leq e^{-|\ln \del_n| j}$ and so $\phi_j(n)\leq e^{-C_n\ve_n j}$ for $C_n\to 1$.
\end{lemma}
\begin{proof}
First, recall that  by Lemma \ref{MixLem} we have $\phi_n(j)\leq \del_n^j$.  
\vskip0.1cm
(i) If $\ve_n\geq c$ for some $c>0$ then $\del_n\leq 1-c$ and so $\phi_n(j)\leq (1-c)^j$.

\vskip0.1cm

(ii) First, we have $\ln\del_n\leq -C_n\ve_n$ for some $C_n\to 1$.
\vskip0.1cm
(1)
 By taking the logarithms of both sides we see that the first item will follow if for each $1\leq j\leq n$ we have 
$C_n\ve_n j-a_n\ln j+\ln B_n\geq 0$. Let us define the function $f_n:[1,n]\to\bbR$ by 
$$
f_n(x)=C_n\ve_n x-a_n\ln x+\ln B_n.
$$
Then $f_n$ has a minimal value at the unique stationary point $x_n=\frac{a_n}{C_n\ve_n}$. Thus $f_n(j)\geq 0$ for all $j\geq1$ if
$$
f_n(x_n)=a_n-a_n\ln (\frac{a_n}{C_n\ve_n})+\ln B_n\geq0
$$  
which is true when $\ln B_n\geq Ca_n\ln (a_n/\ve_n)$ (and $n$ is large enough).
\vskip0.1cm
(2)  By taking the logarithms of both sides if is enough to show that for each $1\leq j\leq n$ we have $-jC_n\ve_n\leq \ln B_n-\ka_nj$ which holds true (for $n$ large enough) when $\ve_n\geq C(\ka_n-\ln B_n/n)$ for some $C>1$.
\end{proof}

\begin{corollary}\label{CC}
(i) If $\del_n\leq 1-c$ for some $c>0$ and $\sup_{n,j}\|\xi_{j,n}\|_{L^\infty}<\infty$ then 
$$
\sup_{t\in\bbR}|\bbP(S_n/\sig_n\leq t)-\Phi(t)|=O(\sig_n^{-1}\ln^3(\sig_n)).
$$

 (ii) Let $0<r_0<1$. We have
 $$
\sup_{t\in\bbR}|\bbP(S_n/\sig_n\leq t)-\Phi(t)|=O(\sig_n^{-(1-2\gamma_n-r_0)})
$$
where 
$$
\gamma_n=\frac{q-1}{a_n(q-2)(2q)^{-1}+2(q-1)}
$$
for some $a_n>1$
in the following situations:
 \vskip0.1cm
 (1) When $K_{q,n}=o(\ve_n^{-1})$ and 
$$
\left(\frac{a_n}{\ve_n}\right)^{3Ca_n}\left(1+K_{\infty,n}\ve_n^{-1}\right)^3\left(|\ln\ve_n|\ve_n^{-1}(K_{2,n}+1)\right)^3\left(1+(a_n-1)^{-1/2}\right)\ll \left(\ve_n \sum_{j=1}^n\text{Var}(\xi_{j,n})\right)^{r_0/2}
$$
for  some constant $C>1$ (where $q_n\ll p_n$ means that $p_n/q_n\to \infty$).
\vskip0.1cm

(2) When $\ve_n \sum_{j=1}^n\text{Var}(\xi_{j,n})\geq Cn^{2\del}$ for some $\del>0$ and $C>0$ (so that $\sig_n\geq cn^{\del}$),  $K_{q,n}=o(\ve_n^{-1})$ and  for some  $q_0\geq q$ we have  
$$
\left(\frac{a_n}{\ve_n}\right)^{3Ca_n}\left(1+K_{q_0,n}n^{\frac 1{q_0}}\ve_n^{-1}\right)^3\left(|\ln\ve_n|\ve_n^{-1}(K_{2,n}+1)\right)^3\left(1+(a_n-1)^{-1/2}\right)\ll n^{\del r_0}.
$$
\vskip0.1cm

(iii) Fix some $0<r_0<1$. We have
 $$
\sup_{t\in\bbR}|\bbP(S_n/\sig_n\leq t)-\Phi(t)|=O\left(\sig_n^{-(1-r_0)}(\ln (\sig_n))^3\right)
$$
in the following situations:
\vskip0.1cm

(3) When $\ve_n\geq C(\ka_n-\ln B_n/n)$ for some $C>1$, for some sequences $\ka_n>0$ and $B_n\geq 1$ (e.g. when $\ve_n\geq C\ka_n$ and $B_n=1$),   $K_{q,n}=o(\ve_n^{-1})$ and 
$$
B_n^3\left(1+K_{\infty,n}\ve_n^{-1}\right)^3\left(|\ln\ve_n|\ve_n^{-1}(K_{2,n}+1)\right)^3\left(1+\ka_n^{-3}\right)\ll \left(\ve_n \sum_{j=1}^n\text{Var}(\xi_{j,n})\right)^{r_0/2}.
$$
\vskip0.1cm

(4) When $\ve_n\geq C(\ka_n-\ln B_n/n)$ for some sequences $B_n\geq 1$, $\ka_n$ and a constant $C>1$, 
 $\ve_n \sum_{j=1}^n\text{Var}(\xi_{j,n})\geq C'n^{2\del}$ for some $\del>0$ and $C'>0$ (so that $\sig_n\geq cn^{\del}$),   $K_{q,n}=o(\ve_n^{-1})$ and for some  $q_0\geq q$ we have  
$$
B_n^3\left(1+K_{q_0,n}n^{\frac 1{q_0}}\ve_n^{-1}\right)^3\left(|\ln\ve_n|\ve_n^{-1}(K_{2,n}+1)\right)^3\left(1+\ka_n^{-3}\right)\ll n^{\del r_0}.
$$

\vskip0.1cm
(iv)  We have
 $$
\sup_{t\in\bbR}|\bbP(S_n/\sig_n\leq t)-\Phi(t)|=O\left(\sig_n^{-1}(\ln (\sig_n))^{4}\right)
$$
in the following  situation:
\vskip0.1cm
(5) When $\ve_n\geq C(\ka_n-\ln B_n/n)$ for some constant $C>1$ and sequences $\ka_n>0$ and $B_n\geq 1$,   $K_{q,n}=o(\ve_n^{-1})$ and 
$$
B_n^3\left(1+K_{\infty,n}\ve_n^{-1}\right)^3\left(|\ln\ve_n|\ve_n^{-1}(K_{2,n}+1)\right)^3\left(1+\ka_n^{-3}\right)\ll \ln\left(\ve_n \sum_{j=1}^n\text{Var}(\xi_{j,n})\right).
$$
\vskip0.1cm

\end{corollary}

\begin{proof}
(i) This result follows from Corollary \ref{Cor BE} together with Lemma \ref{LL} (i).

(ii) Let us take $B_n=\frac{a_n}{\ve_n}$ in the first part of Lemma \ref{LLL} (ii), so that $\phi_n(j)\leq B_nj^{-a_n}$. Recall also that by Lemma \ref{C1} we have $A_n=O\big(|\ln\ve_n|\ve_n^{-1}(K_{2,n}+1)\big)$ and by Lemma \ref{L0}  
we  always have either  $\beta_n=O(1+K_{\infty,n}\ve_n^{-1})$ or $\beta_n=O(1+K_{q_0,n}n^{\frac 1{q_0}}\ve_n^{-1})$.
Now the proofs of items (1) and (2) follow from Theorem \ref{BEthm} (i).

(iii) By the second part of Lemma \ref{LLL} (ii) we have $\phi_n(j)\leq B_ne^{-\ka_n}$. Since we can take either  $\beta_n=O(1+K_{\infty,n}\ve_n^{-1})$ or $\beta_n=O(1+K_{q_0,n}n^{\frac 1{q_0}}\ve_n^{-1})$,  the proofs of items (3) and (4) follow from Theorem \ref{BEthm} (ii).

(iv) The proof is similar to part (iii).

\end{proof}

\begin{example}
Suppose that $\ve_n\geq c n^{-\te}$ for some $\te,c>0$.
\vskip0.1cm
(ii) The condition in item (1) holds true with $a_n=a>1$ when $K_{\infty,n}=o(n^\te)$ and
$$
V(n):=\sum_{j=1}^n\text{Var}(\xi_{j,n})\gg n^{\frac{(21+6a)\te}{r_0}+\te}.
$$ 
Thus, if $V(n)\geq Cn^{\eta}$ for some $\eta>0$ (e.g. when $\text{Var}(\xi_{j,n})\asymp j^{\eta-1}$) then
by taking $q=3$, given $r_0$ and $a$, we get the rate $O(\sig_n^{-(1-w_a-r_0-\ve)})$, $w_a=\frac{12}{24+a}$, for every $\ve$ if $\te$ is small enough.
 \vskip0.1cm
 Similarly, the condition in item (2) holds true with $a_n=a$ if $V(n)\geq Cn^{2\del+\te}$, for some $q_0\geq 3$ we have $K_{q_0,n}=o(n^{\te})$ and 
 $$
3a\te+\frac{3}{q_0}+12\te<\del r_0.
$$
In this case we also get the rate $O(\sig_n^{-(1-w_a-r_0-\ve)})$.
 \vskip0.1cm
 (iii) Let us take $\ka_n=c\ve_n$ for some $c<1$ and $B_n=1$. Suppose that $K_{\infty,n}=o(n^\te)$. Then the conditions in item (3) hold true if $V(n)\geq Cn^{\frac{24\te}{r_0}+\te}$, and we get the rate $O\left(\sig_n^{-(1-r_0)}(\ln (\sig_n))^3\right)$. The conditions of item (4) are in force (and we get the same rate) if $V(n)\geq Cn^{2\del+\te}$, for some $q_0\geq 3$ we have $K_{q_0,n}=o(n^{\te})$ and 
  $$
15\te+\frac{3}{q_0}<\del r_0.
$$
\vskip0.1cm
(iv) Let us suppose that $\ve_n\geq  c(\ln n)^{-\te}$ for some $\te,c>0$ and take $\ka_n=c\ve_n$ for some $c<1$ and $B_n=1$. Then the conditions in item (5) are in force if $K_{\infty,n}=o((\ln n)^{\te})$ and $\ln V(n)\gg (\ln n)^{15\te}$ (e.g. when $15\te<1$ and $c_n (\ln n)^{\te}\geq \text{Var}(\xi_{j,n})\geq cj^{-1}$ with $c_n\to 0$ slower than $(\ln n)^{-\te}$). 
\end{example}

\subsection{Moderate deviations principle}

Let us first formulate a result that follows from taking $\eta=1$ in Theorem \ref{MDPthm} and using our previous estimates on $A_n,j_n$ etc.

\begin{corollary}
(i) If $\sup_n\del_n<1$ (i.e. $\ve_n$ is bounded away from $0$) and $\sup_{n,j}\|\xi_{j,n}\|_{L^\infty}<\infty$ then the MDP and the moment estimates in Theorem \ref{MDPthm}
hold true with $1\leq R_n\leq C$ for some constant $C$. In particular, for every $0<r_0<\frac13$ and all measurable sets $\Gamma\subset\bbR$,
\begin{equation}\label{MdP}
-\inf_{x\in\Gam^o}(x^2/2)\leq \liminf_{n\to\infty}\frac1{\sig_n^{2r_0}}\ln\bbP(S_n/\sig_n^{1+r_0}\in\Gamma)\leq  \limsup_{n\to\infty}\frac1{\sig_n^{2r_0}}\ln\bbP(S_n/\sig_n^{1+r_0}\in\Gamma) \leq-\inf_{x\in\bar\Gam}(x^2/2)
\end{equation}
and for all $p$,
$$
\left|\bbE[S_n^p]-\sig_n^p\bbE[Z^p]\right|=O(\sig_n^{p-1}).
$$

(ii) Suppose that $\ve_n\to0$ as $n\to\infty$ and that $K_{\infty,n}=o(\ve_n^{-1})$. Let $0<c<1$ and $u>1$.  Then for $n$ large enough $\phi_n(j)\leq e^{-c\ve_n j}$ and
 $R_n$ from  Theorem \ref{MDPthm} satisfies
$$
R_n=O(\ve_n^{-3}\ln(\ve_n^{-1}))e^{c^{-u}\ve_n^{-u}}.
$$
Therefore, the condition $R_n=o(\sig_n^{1/3})$ from  Theorem \ref{MDPthm} is satisfied when
$$
q_n:=V(n)\ve_n^{19}|\ln(\ve_n)|^{-6}e^{-6c^{-u}\ve_n^{-u}}\to\infty
$$
where $V(n)=\sum_{j=1}^n\text{Var}(\xi_{j,n})$. In this case we can take any sequence $t_n\to\infty$ so that $t_n=o(\sqrt{q_n})$.

In particular, if $\ve_n\geq c_1n^{-\te_1}$ for some $\te_1,c_1>0$ and 
$$
V(n)\geq n^{19\te_1}\ln ne^{(cc_1)^{-u}n^{u\te_1}}y_n^2
$$
with $y_n\to\infty$ then the MDP holds with $t_n=o(y_n)$.
If $\ve_n\geq c_1(\ln n)^{-\te_1}$ for some $\te_1,c_1>0$ and 
$$
V(n)\geq (\ln n)^{19\te_1}\ln(\ln n)e^{(cc_1)^{-u}(\ln n)^{u\te_1}}y_n^2
$$
with $y_n\to\infty$ then the MDP holds with $t_n=o(y_n)$.
\end{corollary}

\begin{proof}
(i) In the circumstances of part (i) we have  $\phi_n(j)\leq\del_n^j\leq e^{-\ka j}$ for some $\ka>0$. Hence we can take $B_n=1$, $a_n=\ka$ and $\eta=1$ in Theorem \ref{MDPthm}. It also follows from the uniform mixing rates $\phi_n(j)\leq e^{-\ka j}$ and the uniform boundedness of the array that $A_n$ is bounded. Since $\phi(n_0)=\sup_n\phi_n(n_0)<1/2$ for $n_0$ large enough we can take $j_n=n_0$ to be a constant. Hence $R_n$ is bounded and the result follows.

(ii) Let  us take $\ka_n=c\ve_n$ in Lemma \ref{LLL} (ii). Then the mixing conditions in Theorem \ref{MDPthm} hold true with $B_n=1$, $\eta=1$ and $a_n=\ka_n$. Using the estimates on $j_n$ and $A_n$ provided in Corollary \ref{C1} we see that 
$$
R_n=O(\ve_n^{-3}\ln(\ve_n^{-1}))e^{c^{-u}\ve_n^{-u}}.
$$
Now all the results stated in part (ii) follow from applying Theorem \ref{MDPthm} and using \eqref{Var} to bound $\sig_n^2$ from below by $C\ve_n V(n)$.
 
\end{proof}

\subsubsection{Non-uniform-stretched exponential rates}
In this section we will explain how to apply Theorem \ref{MDPthm} with $\eta\not=1$. In order not to overload this section we will not formulate precise results, and instead we will focus on providing conditions for non-uniform mixing rates of the form $\phi_j(n)\leq B_ne^{-a_nj^\eta}$ for $B_n,a_n>0$ and positive $\eta\not=1$ (which we think is interesting by its own right). After this is established explicit conditions can be given, for instance, in two cases $\ve_n\geq cn^{-\te}$ and $\ve_n\geq c (\ln n)^{-\te}$.

\begin{lemma}\label{LMDP}
(i) Let $\eta\in(0,1)$  and let $a_n$ be a positive sequence.
Set  $x_n=(a_n\eta /|\ln \del_n|)^{\frac{1}{1-\eta}}$.
If $x_n\leq 1$  then
$\phi_n(j)\leq \del_n e^{a_n}e^{-a_n j^\eta}$ for every $1\leq j\leq n$. Otherwise, 
 $\phi_n(j)\leq \max(\del_n e^{a_n},U_n)e^{-a_n j^\eta}$
where
\begin{equation}\label{EQ}
U_n=\exp\left(a_n^{\frac{1}{1-\eta}}|\ln\del_n|^{-\frac{\eta}{1-\eta}}\left(\eta^{\frac{\eta}{1-\eta}}-\eta^{\frac1{1-\eta}}\right)\right).
\end{equation}

\vskip0.1cm
(ii) Let $\eta>1$ and let $a_n$ be a positive sequence. Then
$\phi_n(j)\leq \del_n e^{a_n}e^{-a_n j^\eta}$ for every $1\leq j\leq n$ if $x_n\geq n$ (where $x_n$ was defined above).
\end{lemma}

\begin{proof}
Let $B_n>0$.
By Lemma \ref{MixLem} we have $\phi_n(j)\leq \del_n^j$. Now $\del_n^j\leq B_ne^{-a_n j^\eta}$ for every $1\leq j\leq n$  if and only if 
$$
f_n(j):=j|\ln \del_n|+\ln B_n-a_nj^{\eta}\geq0
$$
for every $1\leq j\leq n$. Let us consider $f_n(x)$ as a function of $x>0$. Suppose next that $0<\eta<1$ and take $B_n\geq \del_n e^{a_n}$ (so that $f_n(1)\geq 0$).
Then $f_n''(x)>0$ and $f_n'(x)=0$ iff $x=x_n$. Hence $x_n$ is the absolute minimum of $f_n$. 
 If $x_n\leq 1$ then $f_n$ is increasing on $[1,\infty)$ and so $f_n(j)\geq 0$ for all $j\in\bbN$. Hence we can take $B_n=\del_n e^{a_n}$. If $x_n\geq 1$ then we notice  that  $f_n(x_n)\geq0$ iff $B_n\geq U_n$. Hence, for $B_n=\max(\del_n e^{a_n}, U_n)$ we see that $f_n(x_n)\geq 0$ and so $f_n\geq0$.

Now, suppose that $\eta>1$. Then $f_n''(x)<0$ and so $x_n$ is a global maximum.  Thus, $f_n$ is increasing on $[1,x_n]$ and decreasing on $[x_n,\infty)$. Let us now take $B_n=\del_ne^{a_n}$ so that $f_n(1)=0$. Hence, if $x_n\geq n$ then $f_n(j)\geq0$ for all $1\leq j\leq n$.
\end{proof}

\begin{remark}
By analyzing the function $f_n$
certain rates can be obtained also  when $\eta>1$ and $x_n\leq n$, but this leads to ineffective rates of the form $\phi_n(j)\leq Ce^{-a_n j^\eta}$ with $a_n\leq cn^{-\eta}$.
\end{remark}

\begin{example}[Explicit stretched exponential mixing rates]
(i) Let $\eta\in(0,1)$.
\vskip0.1cm
(1) Let $a_n=a>0$ be a constant sequence. Then $x_n\to\infty$. Suppose that $\ve_n\geq c_1n^{-\te_1}$ for some $c_1,\te_1>0$, then $|\ln\del_n|\leq C_n c_1n^{-\te_1}$ with $C_n\to1$ and so
$$
\phi_j(n)\leq CU_ne^{-aj^\eta}=Ce^{C_nc_1 c_{\eta,a}n^{\frac{\te_1\eta}{1-\eta}}}e^{-aj^\eta}
$$
with  $c_{a,\eta}=a^{\frac1{1-\eta}}\left(\eta^{\frac{\eta}{1-\eta}}-\eta^{\frac1{1-\eta}}\right)$. These rates are effective when $\te<1-\eta$ since otherwise the term $n^{\frac{\te\eta}{1-\eta}}$ dominates all the powers $j^\eta$.

If $\ve_n\geq c(\ln n)^{-\te}$ for some $c,\te>0$. Then $|\ln\del_n|\leq C_nc(\ln n)^{-\te}$ with $C_n\to1$ and 
$$
\phi_j(n)\leq CU_ne^{-aj^\eta}=Ce^{C_ncc_{\eta,a}(\ln n)^{\frac{\te_1\eta}{1-\eta}}}e^{-aj^\eta}.
$$
\vskip0.1cm

(2) Let us take $a_n=|\ln\del_n|/\eta$ so that $x_n=1$. Then 
$$
\phi_n(j)\leq e^{a_n}e^{-a_nj^\eta}.
$$
When $\ve_n\geq c_1n^{-\te_1}$ we get that 
$$
\phi_n(j)\leq Ce^{-C_n c_1n^{-\te}j^\eta}, C_n\to 1.
$$
This is an effective non-uniform mixing rate when $\te_1<\eta$ (for otherwise the term $n^{\te_1}$ will dominate $j^{\eta}$). If, instead,
$\ve_n\geq c(\ln n)^{-\te}$  then
$$
\phi_n(j)\leq Ce^{-a_nj^\eta}=Ce^{C_nc\eta^{-1}(\ln n)^{-\te}j^\eta}.
$$
\vskip0.1cm

(3) When $x_n\geq 1$ we write $a_n=|\ln\del_n|v_n/\eta$ with $v_n\geq1$ and $C_n\to1$. Then 
$$
U_n=e^{v_n^{1/\eta}(\eta^{-1}-1)|\ln \del_n|}
$$
and so, if $v_n$ is large enough then $U_n\geq e^{a_n}$ and when  $\ve_n\geq c_1n^{-\te_1}$ we get
$$
\phi_n(j)\leq e^{C_n c_1 n^{-\te_1}(1-\eta)v_n^{\frac{\eta}{1-\eta}}}e^{-C_n c_1v_n n^{-\te_1}j^{\eta}}
$$
Thus, when $v_n=n^{\zeta}$ for some $\zeta<\min(\te_1, \frac{(1-\eta)\te_1}{\eta})$ then
$$
\phi_n(j)\leq Ce^{-C_n n^{\zeta-\te_1}j^\eta}
$$
which is effective if also $\zeta<\eta+\te_1$.

If, instead, $\ve_n\geq c(\ln n)^{-\te}$  and $v_n^{1/\eta}=(\ln n)^{w}$, $w<\te$  we get 
$$
\phi_n(j)\leq Ce^{-C_nc(\ln n)^{\eta w-\te}j^\eta}.
$$
\vskip0.1cm

(ii) Let $\eta>1$. If $\ve_n\geq cn^{-\te}$ then the condition $x_n\geq n$ holds true iff $a_n\leq \eta^{-1}(C_n c_1)^{\eta-1}n^{-(\te+\eta-1)}$. Let $a_n=Cn^{-\zeta}$ with $\zeta\geq \te+\eta-1$  (when $\zeta=\te+\eta-1$ 	we assume $C<c_1^{\eta-1}\eta^{-1}$). Then, since $a_n=o(1)$, 
$$
\phi_n(j)\leq C_1e^{Cn^{-\zeta}j^\eta}
$$
which is effective when also $\zeta<\eta$ (which is possible when $\eta>\te+\eta-1$, namely $\te<1$).
\end{example}

\section{First step of the proofs: regular blocks for non-uniformly $\al$-mixing arrays}\label{Block}


\subsubsection{An overview}
In this section we will essentially make a reduction to the case when $\sig_n^2$ grows sufficiently fast in $n$. More precisely, we will decompose $S_n$ into blocks $S_n=X_{1,n}+...+X_{k_n,n}$ with $k_n\asymp \sig_n^2/A_n^2$ so that $X_{j,n}$ is a function of $\xi_{k,n}$ for $k\in B_j(n)$, where the blocks $B_{j}(n)$ are ordered so that $B_j(n)$ is to the left of $B_{j+1}(n)$.
A similar decomposition was established in \cite{HafFCLT} for uniformly mixing triangular arrays (with $A_n$ being bounded), and in this section we will extend it to non-uniformly mixing arrays (which, as opposed to \cite{HafFCLT}, are not necessarily uniformly bounded in some $L^q$). The main difference here is that the variance $\sig_n^2$ does not necessarily grow linearly fast in the new number of summands $k_n\asymp A_n^2/\sig_n^2$ since in general $A_n$ is unbounded, but since we assume that $A_n=o(\sig_n)$ we still get that the number of summands is not negligible in comparison with the variance, which will be enough for our methods of proof to be effective.
\begin{remark}\label{R beta 2}
We stress that (as noted in Remarks \ref{R beta} and \ref{R beta 1}) when $c_1 n\leq \sig_n^2\leq c_2n$ for some constants $c_1,c_2>0$ and all $n$ large enough then we  can skip all the results stated in this section, set $X_{j,n}=\xi_{j,n}$, $\beta_n=K_{q,n}$ and $A_n=1$ and proceed with the proof as in the next sections.
\end{remark}

\subsubsection{First steps towards the construction of the blocks}
Let us now start with the construction of the blocks $X_{j,n}$.
Set
$$
\del_n(m)=\sup_{k}\sum_{s=m}^{n-k}\left\|\bbE[\xi_{k+s,n}|\xi_{k,n},...,\xi_{1,n}]\right\|_{L^2}.
$$
We first need the following result.
\begin{lemma}\label{lemm del}
Suppose that $\sup_{j}\|\xi_{j,n}\|_{L^q}\leq K_{q,n}<\infty$ for some $q>2$. Then
$$
\del_n(m)\leq 2K_{q,n}\sum_{j=m}^{n}(\al_n(j))^{1/2-1/q}.
$$
\end{lemma}\label{al lem}
\begin{proof}
By \eqref{RRel} we have 
$$
\al(\cG,\cH)=\frac14\varpi_{\infty,1}(\cG,\cH),
$$
where $\varpi_{\infty,1}(\cdot,\cdot)$ is defined in \eqref{var pi def} and $\al(\cdot,\cdot)$ in \eqref{al def1}. By applying \eqref{MixCoe} we see that
$$
\|\bbE[\xi_{k+s,n}|\xi_{1,n},...,\xi_{k,n}]\|_{L^2}\leq K_{q,n}\varpi_{q,2}(\cG,\cH)\leq 2K_{q,n}\big(\al(\cG,\cH)\big)^{1/2-1/q}\leq K_{q,n}2 \big(\al(s)\big)^{1/2-1/q}
$$
where $\cH$ is the $\sig$-algebra generated by $\zeta_{k+s,n}$ and $\cG$ is the $\sig$-algebra generated by $\{\zeta_{1,n},...,\zeta_{k,n}\}$.
\end{proof}
Since $(\al_n(j))^{1/2-1/q}\leq (\al_n(j))^{1/p-1/q}$ we see that $\del_n(m)\leq 2K_{q,n}\Gamma_n(m)$, where $\Gamma_n(\cdot)$ was defined in \eqref{Gamma n def}.
Hence,
$$r_n=\max\left(\Gamma_n^{-1}(\frac1{8 K_{q,n}}),1\right)$$
 satisfies
$$
\del_n(r_n)<\frac14.
$$


\subsubsection*{\textbf{The blocks}}
Fix some $n$ and set $S_k=S_k^{(n)}$.
Let $A_n$ be given by \eqref{A_n def}. Set $\ve_0=\frac1{18}$. Then $A_n\ve_0=r_n(1+K_{2,n})$, where $K_{2,n}=\max_{j}\|\xi_{j,n}\|_{L^2}$.
 Let us take $b_1=b_1(n)$ to be the first time that 
$$\|S_{b_1}\|_{2}\geq A_n.$$ Set $Y_1=Y_1(n)=S_{b_1}$. 
Note that 
$$
A_n\leq \|S_{b_1}\|_{2}\leq K_{2,n}b_1
$$
and so $b_1\geq A_n/K_{2,n}\geq \ve^{-1}r_n\geq r_n$.
Next, we take $\beta_1=\beta_1(n)$ to be the smallest positive integer so that 
$$\|S_{b_1+\beta_1}-S_{b_1}\|_{2}\geq A_n\ve_0.$$ 
Set $Z_1=Z_1(n)=S_{b_1+\beta_1}-S_{b_1}$. 
Note that 
$$
\ve_0 A_n\leq \|S_{b_1+\beta_1}-S_{b_1}\|_{2}\leq K_{2,n}\beta_1
$$
and so $\beta_1\geq \ve_0 A_n/K_{2,n}\geq r_n$.
Continuing this way we get blocks $Y_1(n),Z_1(n),Y_2(n),Z_2(n),...$ of the form $\sum_{j\in I}\xi_{j,n}$ for an interval $I$ so that: 
\begin{enumerate}
\item the size of the gap between two consecutive $Y_j(n)$'s is at least $r_n$;
\vskip0.1cm
\item  the size of the gap  between  between two consecutive $Z_j(n)$'s is at least $r_n$;
\vskip0.1cm
\item $$A_n\leq \|Y_j(n)\|_2\leq A_n+K_{2,n},\,\,A_n\ve_0 \leq \|Z_j(n)\|_2\leq A_n\ve_0+K_{2,n};$$
\vskip0.1cm
\item $\sum_{j=1}^{k_n}(Y_j(n)+Z_j(n))=S_n=\sum_{j=1}^{n}\xi_{j,n}$, where $k_n$ is the number of $Y_j(n)$'s and $Z_j(n)$'s. 
\end{enumerate}
Note that $\ve_0 A_n\geq r_n\geq1$ and so  both $\|Y_{j}(n)\|_2$ and $\|Z_j(n)\|_2$ are at least $1$.
Let $X_j(n)=X_{j,n}=Y_j(n)+Z_j(n)$ and let 
 us denote by $B_j(n)=\{a_j(n),...,b_j(n)\}$ the set of indexes so that 
\begin{equation}\label{B j n}
X_j(n)=\sum_{k\in B_j(n)}\xi_{k,n}.
\end{equation}
Notice that 
$$
\|X_j(n)\|_2\geq \|Y_j(n)\|_{2}-\|Z_j(n)\|_2\geq A_n-\ve_0 A_n-K_{2,n}\geq (1-2\ve_0)A_n
$$
where we used that $\ve_0 A_n\geq r_n K_{2,n}\geq K_{2,n}$. Using also the maximality property in the construction of the blocks $Y_{j}(n)$ and $Z_{j}(n)$ and that $K_{2,n}\leq A_n$ we conclude that
\begin{equation}\label{Vars}
A_n(1-2\ve_0)\leq \|X_j(n)\|_{L^2}\leq \max_{m\leq b_{j}(n)}\left\|\sum_{k=a_j(n)}^{m}\xi_{j,n}\right\|_2\leq 4A_n.
\end{equation}
We note that in the above construction we might technically need to absorb the last block in the penultimate block, but this only amounts to replacing $A_n$ with $2A_n$, which will make no difference in the following arguments.
\vskip0.1cm

The main additional property of the blocks $X_{j}(n)$ is the following non-uniform version of  \cite[Proposition 3.6]{HafFCLT} (which was only valid for uniformly mixing and bounded arrays.)
\begin{proposition}\label{Prop.1}
For every $1\leq s_1<s_2\leq k_n$ we have 
$$
\frac{1}{64}\sum_{j=s_1}^{s_2}\text{Var}(X_j)\leq \text{Var}\left(\sum_{j=s_1}^{s_2}X_j\right)\leq \frac{9c_{\ve_0}^2}{4}\sum_{j=s_1}^{s_2}\text{Var}(X_j)
$$
where $X_j=X_j(n)$ and $c_{\ve_0}=c_{\ve_0}=\frac{1+\ve_0}{1-2\ve_0}=\frac{19}{16}$. As a consequence, by \eqref{Vars} the number of blocks $k_n$ satisfies $k_n\asymp \sig_n^2/A_n^2$.
\end{proposition}
The proof of Proposition \ref{Prop.1} is based on the following two results.

\begin{lemma}\label{SumVars.1}
For all $1\leq s_1<s_2\leq k_n$ we have
$$
\frac12\sum_{j=s_1}^{s_2}\text{Var}(Y_j)\leq \text{Var}\left(\sum_{j=s_1}^{s_2}Y_j\right)\leq \frac32\sum_{j=s_1}^{s_2}\text{Var}(Y_j)
$$
and 
$$
\frac12\sum_{j=s_1}^{s_2}\text{Var}(Z_j)\leq \text{Var}\left(\sum_{j=s_1}^{s_2}Z_j\right)\leq \frac32\sum_{j=s_1}^{s_2}\text{Var}(Z_j)
$$
where $X_j=X_j(n)$ and $Y_j=Y_j(n)$.
\end{lemma}
and 
\begin{lemma}\label{2Lemm.1}
For all $1\leq s_1<s_2\leq k_n$  we have 
$$
\left|\frac{\text{Var}(\sum_{j=s_1}^{s_2}X_j)}{\text{Var}(\sum_{j=s_1}^{s_2}Y_j)}-1\right|\leq \cD(\ve_0)=
 \left(12\ve_0^2+2\sqrt{12}\ve_0\right)\leq\frac12.
$$
\end{lemma}

\begin{proof}[Proof of Proposition \ref{Prop.1} based on Lemmas \ref{SumVars.1} and \ref{2Lemm.1}]
First, by applying Lemma \ref{2Lemm.1} and then Lemma  \ref{SumVars.1}  we get that
\begin{equation}\label{Uppp}
\frac14 \sum_{j=s_1}^{s_2}\text{Var}(Y_j(n))\leq \text{Var}\left(\sum_{j=s_1}^{s_2}X_j(n)\right)\leq \frac94 \sum_{j=s_1}^{s_2}\text{Var}(Y_j(n)).
\end{equation}
Next, notice that $\|Y_j(n)\|_2\leq A_n+K_{2,n}\leq (1+\ve_0)A_n=c_{\ve_0}(1-2\ve_0)A_n$
where $c_{\ve_0}=\frac{1+\ve_0}{1-2\ve_0}$. Using also \eqref{Vars} we see that $\text{Var}(Y_{j}(n))\leq c_{\ve_0}^2\text{Var}(X_j(n))$. Now, recall that by \eqref{Vars} we have $\|X_j(n)\|_{L^2}\leq 4A_{n}$. Since $\|Y_{j}(n)\|_2\geq A_n$ we see that $\text{Var}(Y_{j}(n))\geq \frac1{16}\text{Var}(X_{j}(n))$. The proof of the proposition is completed by combining the above upper and lower bounds on $\text{Var}(Y_{j}(n))$ with \eqref{Uppp}.
\end{proof}

\begin{proof}[Proof of Lemma \ref{SumVars.1}]
Let us prove the first estimate. First,
$$
\text{Var}\left(\sum_{j=s_1}^{s_2}Y_j\right)=\sum_{i=s_1}^{s_2}\|Y_i\|_{L^2}^2+2\sum_{s_1\leq i<j\leq s_2}\text{Cov}(Y_i,Y_j).
$$
Since the size of the gap between $Y_i$ and $Y_{i+1}$ is at least $r_n$ and $\|Y_i\|_{2}\geq 1$, we have
\begin{equation}\label{SimTo.1}
2\left|\sum_{s_1\leq i<j\leq s_2}\text{Cov}(Y_i,Y_j)\right|\leq 2\sum_{s_1\leq i<s_2}\left|\bbE[Y_i\sum_{j>i}Y_j]\right|
\end{equation}
$$
=2\sum_{s_1\leq i<s_2}\left|\bbE\left[Y_i\bbE\big[\sum_{j>i}Y_j|Y_i\big]\right]\right|\leq 
2\sum_{s_1\leq i<s_2}\|Y_i\|_2\left\|\bbE\big[\sum_{j>i}Y_j|Y_i\big]\right\|_2$$$$
\leq 2\del_n(r_n)\sum_{s_1\leq i<s_2}\|Y_i\|_2\leq \frac12\sum_{s_1\leq i<s_2}\|Y_i\|_2^2.
$$
The proof for the $Z_j$'s is similar.
\end{proof}

\begin{proof}[Proof of Lemma \ref{2Lemm.1}]
First, notice that $\ve_0=\frac1{18}$ indeed satisfies $\cD(\ve_0)\leq \frac12$.
Denote $A_n=A$ and $K_{2,n}=K$.
We have 
$$
\text{Var}\left(\sum_{j=s_1}^{s_2}X_j\right)=\text{Var}\left(\sum_{j=s_1}^{s_2}Y_j\right)+\text{Var}\left(\sum_{j=s_1}^{s_2}Z_j\right)+2\text{Cov}\left(\sum_{j=s_1}^{s_2}Y_j,\sum_{j=s_1}^{s_2}Z_j\right).
$$
Observe that $\|Z_j\|_2\leq A\ve_0+K\leq 2A\ve_0\leq 2\ve_0\|Y_j\|_2$, and so 
by Lemma \ref{SumVars.1} we have 
$$
\text{Var}\left(\sum_{j=s_1}^{s_2}Z_j\right)\leq \frac32\sum_{j=s_1}^{s_2}\text{Var}(Z_j)\leq 
\frac{3\cdot 4\ve_0^2}{2}\sum_{j=s_1}^{s_2}\text{Var}(Y_j)\leq (12\ve_0^2)\text{Var}\left(\sum_{j=s_1}^{s_2}Y_j\right).
$$
Thus, using also the Cauchy-Schwartz inequality we have
$$
\left|\text{Var}\left(\sum_{j=s_1}^{s_2}X_j\right)-\text{Var}\left(\sum_{j=s_1}^{s_2}Y_j\right)\right|
\leq (12\ve_0^2)\text{Var}\left(\sum_{j=s_1}^{s_2}Y_j\right)+2\left\|\sum_{j=s_1}^{s_2}Z_j\right\|_{L^2}\left\|\sum_{j=s_1}^{s_2}Y_j\right\|_{L^2}
$$
$$
\leq \left(12\ve_0^2+2\sqrt{12\ve_0^2}\right)\text{Var}\left(\sum_{j=s_1}^{s_2}Y_j\right).
$$
\end{proof}

\section{The functional CLT with rates via Martingale approximation}\label{SecFunc}
Let $A_n$ be defined by \eqref{A_n def}.
Let us write $X_j(n)=X_{j,n}$. Then $\{X_{j,n}:\,1\leq j\leq k_n\}$ is a new triangular array. The idea in the proof of Theorems \ref{FunCLT} and \ref{FuncCLT MC} is to obtain functional CLT rates for the new array using martingale approximation and functional Berry-Esseen bounds for martingales. However, we first need to obtain estimates on the Prokhorov distance between the random functions corresponding to the new and original arrays. 

The structure of this section is as follows. In Section \ref{Sec5.1} we will state a few general results which will be in constant use in the course of the proof of the main results. In Section \ref{Sec5.2} we will state a few auxiliary results concerning the new array $\{X_{j,n}\}$. In Section \ref{Initial} we will make the aforementioned  reduction of the functional CLT rates to the new array $\{X_{j,n}\}$. In Section \ref{M approx} we will present the martingale approximation for the function $\cW_n(t)$ corresponding to the new array, while in Section \ref{F approx} we will provide general upper bounds on the Prokhorov distance $d_P(M_n, B)$, where $M_n$ is the approximating martingale. The rest of the sections are dedicated to estimating the quadratic variation of the martingale $M_n$ (which, as usual, is the more technical part of the proof) and complete the proof of Theorems \ref{FunCLT} and \ref{FuncCLT MC} and Corollaries \ref{Corr1} and \ref{Corr2}.

\subsection{General auxiliary results}\label{Sec5.1}
Before we begin we need the following simple lemma.
\begin{lemma}\label{MaxLem}
Let $Z_1,Z_2,...,Z_k$ be real-valued random variables defined on a common probability space. Let $M=\max\{|Z_1|,|Z_2|,...,|Z_k|\}$. Then for every $p\geq1$ we have 
$$
\|M\|_p\leq k^{\frac1{p}}\max\{\|Z_j\|_{p}:\,1\leq j\leq k\}.
$$
\end{lemma}
\begin{proof}
We have $|M|^p\leq\sum_{k=1}^k|Z_j|^p$, and the lemma follows by taking expectation of both sides.
\end{proof}

\begin{lemma}\label{ProkAssLemma}
Let $Q_1(t),Q_2(t)$ be two random functions so that $Q_1(\cdot),Q_2(\cdot)\in D[0,1]$. Let $\cL_i$ be the law of $Q_i$, $i=1,2$. Then for every $q\geq1$,
$$
d_P(Q_1,Q_2)=d_P(\cL_1,\cL_2)\leq\left\|\sup_{t\in[0,1]}|Q_1(t)-Q_2(t)|\right\|_q^{\frac{q}{q+1}}
$$
where $d_P$ is the Prokhorov metric on $D[0,1]$.
\end{lemma}
\begin{proof}
First, it follows from the definition of $d_P$ that $d_P(Q_1,Q_2)\leq \ve_0$ if $$\bbP(\sup_{t\in[0,1]}|Q_1(t)-Q_2(t)|\geq \ve_0)\leq \ve_0.$$ By the Markov inequality we have 
$$
\bbP(\sup_{t\in[0,1]}|Q_1(t)-Q_2(t)|\geq \ve_0)\leq\left\|\sup_{t\in[0,1]}|Q_1(t)-Q_2(t)|\right\|_q^q\ve_0^{-q}
$$
and the lemma follows by taking $\ve_0=\left\|\sup_{t\in[0,1]}|Q_1(t)-Q_2(t)|\right\|_q^{\frac{q}{q+1}}$.
\end{proof}

In the course of the proof of Theorems \ref{FunCLT} and \ref{FuncCLT MC} we will use several times the following result \cite[Proposition 7]{MPU0}:
\begin{proposition}\label{Prp7}
Let $X_1,...,X_m$ be square integarble random variables, and set $\mathscr F_i=\sig\{X_1,...,X_i\}$. Let us fix some $p\geq2$ and set 
$$
b_i=\max_{i\leq l\leq m}\left\|X_i\sum_{k=i}^{l}\bbE[X_k|\mathscr F_i]\right\|_{p/2}.
$$
Set also $S_k=\sum_{j=1}^k X_j$. Then 
$$
\|S_m\|_p\leq\left(2p\sum_{i=1}^{m}b_i\right)^{1/2}
$$
and with $M=\max_{1\leq k\leq m}|S_k|$, 
$$
\|M\|_p\leq C_p\left(\sum_{i=1}^{m}b_i\right)^{1/2}
$$
where $C_2=16$ and $C_p=\left(1-2^{(1-p)/2p}\right)^{-2p}(2p)^{p/2}$ for $p>2$.
\end{proposition} 
We note that for martingales this proposition essentially reduces to a combination of Doob's maximal inequality and the Burkholder inequality (note that in this case $b_i=X_i^2$).

\subsection{Auxiliary results related to the new array $\{X_{j,n}\}$}\label{Sec5.2}
Let us write $X_j(n)=X_{j,n}=\sum_{k=a_j(n)}^{b_j(n)}\xi_{k,n}$.
 Then, by \eqref{Vars} we have $\|X_{j,n}\|_2\leq 4A_n$, and so by  \eqref{MaxMom}  for each $j$ and $n$,
\begin{equation}\label{Then}
\|X_{j,n}\|_{q}\leq 4\beta_n A_n.
\end{equation}
Next, let
$$
\varpi_{q,p,n}(k)=\sup\{\varpi_{q,p}(\cF_n(s),\cF_n(s+k,n)): s\leq n-k\}
$$
where $\varpi_{q,p}$ are the mixing coefficient defined in \eqref{var pi def}, and the $\sig$-algebras $\cF_{n}(s)$ and $\cF_n(s+k,n)$ are defined after \eqref{al def}.
 Then by \eqref{MixCoe}, 
\begin{equation}\label{Var po rel}
\varpi_{q,p,n}(k)\leq 2\left(\al_n(k)\right)^{1/p-1/q}.
\end{equation}

In the course of the proof of Theorems \ref{FunCLT} and \ref{FuncCLT MC} we will need the following result.
\begin{lemma}\label{Norm Con}

There is a constant $C=C_{p}>0$ so that for all indexes $0<u_1<u_2<...<u_l\leq k_n$, $l\in\bbN$ we have 
\begin{equation}\label{Norm Con Need}
\left\|\max_{1\leq s\leq l}\Big|\sum_{j=1}^{s}X_{u_j,n}\Big|\right\|_{p}\leq C_pQ_n\sqrt{l}
\end{equation}
where $Q_n=\sqrt{\Gamma_n(1)}A_n\beta_n$ (and $\Gamma_n$ and $A_n$ are defined in \eqref{Gamma n def} and \eqref{A_n def}, while $\beta_n$ comes from \eqref{MaxMom}).

\end{lemma}
\begin{proof}
Let $b_i=b_{i,n}(\{u_j\})=\|X_{u_i,n}\sum_{j=i}^{l}\bbE[X_{u_j,n}|X_{u_1,n},...,X_{u_i,n}]\|_{p/2}$. Then by the H\"older inequality and the definition of the mixing coefficients $\varpi_{p,q}$ and \eqref{Var po rel} we have
$$
\|b_i\|_{p}\leq\|X_{u_i,n}\|_{p}\max\{\|X_{k,n}\|_{q}\}\left(1+\sum_{k=1}^{k_n}\varpi_{p,q,n}(k)\right)\leq C(\beta_n A_n)^2(1+\sum_{k}(\al_n(k))^{1/p-1/q})$$$$\leq 8(\beta_n A_n)^2(1+\Gamma_n(1))
$$
where we have also used that $\|X_{j,n}\|_q=O(A_n\beta_n)$,
Now the lemma follows from Proposition \ref{Prp7}.
\end{proof}
\begin{remark}
As noted in Remark \ref{R beta 2}, when $\sig_n^2$ grows linearly fast in $n$ we set  $A_n=1$,  $\beta_n=K_{q,n}$ and $X_{j,n}=\xi_{j,n}$. Then (as can be seen from its proof) Lemma \ref{Norm Con} remains true also in that case.
\end{remark}

 Next, set $\sig_{k,n}=\|S_{k}^{(n)}\|_2=\|\sum_{j=1}^{k}\xi_{j,n}\|_2$, $\sig_n=\sig_{n,n}$ and for every $t\in[0,1]$ set
 $$v_n(t)=\inf\{1\leq k\leq n:\,\sig_{k,n}^2\geq \sig_n^2t\}.$$
 Let $B_{j_n(t)}(n)$ be the block (recall \eqref{B j n})  so that $v_{n}(t)\in B_{j_n(t)}(n)$.
 We will also use the following result.
\begin{lemma}\label{Lemm0}
There is A constant $C>0$ so that for all $t\in[0,1]$,
$$
|\sig_{b_{j_n(t)},n}^2-t\sig_{n}^2|\leq CA_n^2(1+K_{p,n}^2)\Gamma_n(1).
$$
\end{lemma}
\begin{proof}
Because Lemma \ref{lemm del} we have 
$$\max_k|\sig_{k,n}^2-\sig_{k-1,n}^2|=\max_k|\text{Var}(\xi_{k,n})-2\text{Cov}(\xi_{k,n},S_{k-1}^{(n)})|\leq K_{2,n}^2+2K_{p,n}^2\sum_{j\geq 1}(\al_n(j))^{1/p-1/q}:=U_n$$
Taking $k=v_n(t)$ we have $\sig_{k,n}^2\geq t\sig_n^2>\sig_{k-1,n}^2\geq \sig_{k,n}^2-U_n$ 
and so
$$
|\sig_{v_n(t),n}^2-t\sig_n^2|\leq U_n.
$$

Next, let  $a_j(n)$ be the left end point of $B_j(n)$. Then for every $j$ and $m\in B_j(n)$  we have 
$$
\left\|S_{b_j(n)}^{(n)}-S_m^{(n)}\right\|_{2}^2=\left\|X_{j,n}-\sum_{j=a_j(n)}^{m}\xi_{j,n}\right\|_{2}\leq \|X_{j,n}\|_2+\left\|\sum_{j=a_j(n)}^{m}\xi_{j,n}\right\|_{2}.
$$
Now, by \eqref{Then} have $\text{Var}(X_{j,n})=\|X_{j,n}\|_2^2\leq 16A_n^2$ and by \eqref{Vars} we also have $\|\sum_{j=a_j(n)}^{m}\xi_{j,n}\|_{2}\leq C_0A_n^2$ for some constant $C_0$, and so 
\begin{equation}\label{I.1}
\max_{j}\max_{m\in B_j(n)}\|S_{b_j(n)}^{(n)}-S_m^{(n)}\|_{2}^2\leq CA_n^2.
\end{equation}

Finally, since $v_n(t)\in B_{j_n(t)}$ using 
 we obtain that 
$$|\sig_{b_{j_n(t)},n}^2-\sig_{v_n(t),n}^2|\leq \left\|S_{b_{j_n(t)}}^{(n)}-S_{v_n(t)}^{(n)}\right\|_2^2+2\left|\text{Cov}(S_{b_{j_n(t)}}^{(n)}-S_{v_n(t)}^{(n)},S_{v_n(t)}^{(n)})\right|$$
$$
=O(A_n^2+A_n\del_n(1))
$$
and the lemma follows using also Lemma \ref{lemm del}.
\end{proof}

\subsection{Initial approximation: reduction to the new array $\{X_{j,n}\}$}\label{Initial}
In this section we will complete the final step of the approximation of  $W_n$ by the function $\cW_n$ generated by the new array $\{X_{j,n}\}$. 
Let us define 
$$
\cW_{n}(t)=\sig_n^{-1}\sum_{j=1}^{b_{j_n(t)}(n)}\xi_{j,n}=\sig_n^{-1}\sum_{u=1}^{j_n(t)}X_{u,n}.
$$

\begin{lemma}\label{ApproxLemma}
We have
$$
\left\|\sup_{t\in[0,1]}|W_{n}(t)-\cW_{n}(t)|\right\|_{q}\leq C A_n^{1-2/q}\beta_n\sig_n^{-(1-2/q)}.
$$
\end{lemma}
\begin{proof}
For each $t\in[0,1]$ there is a unique $1\leq k\leq k_n$ so that $v_n(t)\in B_k(n)$, and in this case $j_n(t)=k$. Thus,
\begin{equation}\label{I}
\sup_{t\in[0,1]}|\cW_{n}(t)-W_{n}(t)|
\leq\sig_n^{-1}\max\{Z_{k,n}: 1\leq k\leq k_n\}
\end{equation}
where 
$$
Z_{k,n}=\max\{|S_{b_k(n)}^{(n)}-S_m^{(n)}|: m\in B_k(n)\}. 
$$
By \eqref{MaxMom} and \eqref{I.1} we have
\begin{equation}\label{II.1}
\max\{\|Z_{k,n}\|_{q}:\,1\leq k\leq k_n\}\leq C\beta_n A_n.
\end{equation}
 Now, by Lemma \ref{MaxLem},
we have
$$
\left\|\max\{|Z_{k,n}|: 1\leq k\leq k_n\}\right\|_{q}\leq (k_n)^{1/p}\max\{\|Z_{k,n}\|_{q}: 1\leq k\leq k_n\}.
$$
Finally, by Proposition \ref{Prop.1} we have $k_n\leq C\sig_n^2/A_n^2$ and
 the lemma follows by \eqref{I}, \eqref{II.1} and the above inequality.
\end{proof}
\begin{corollary}\label{Cor Approx}
We have
\begin{equation}\label{CorInq}
d_P(W_n,\cW_n)\leq C\left(A_n/\sig_n\right)^{\frac{q-2}{q+1}}\beta_n^{\frac{q}{q+1}}
\end{equation}
where $d_P$ is the Prokhorov metric of probability laws on the Skorokhod space $D[0,1]$.
\end{corollary}

\begin{proof}
The corollary follows from the combinations of Lemma \ref{ApproxLemma} and Lemma \ref{ProkAssLemma}.
\end{proof}

\subsection{Martingale approximation (for the new array)}\label{M approx}
In view of Corollary \ref{Cor Approx}, our goal is to estimate $d_P(\cW_n,B)$, where $B$ is a standard  Brownian motion. To establish that let us first present a certain type of   martingale-coboundary representation of the sums $S_m^{(n)}=\sum_{j=1}^m\xi_{j,n}$. 
Set 
$$
d_{j,n}=\xi_{j,n}+R_{j,n}-R_{j-1,n},\,\,R_j=\sum_{n\geq s\geq j+1}\bbE[\xi_{s,n}|\zeta_{1,n},\zeta_{1,n},...,\zeta_{j,n}].
$$
Let us also set 
$$
D_{j,n}=\sum_{k\in B_j(n)}d_{k,n}=X_{j,n}+R_{b_{j}(n),n}-R_{a_j(n)-1,n}
$$
where $B_j(n)=\{a_j(n),...b_j(n)\}$ are defined in \eqref{B j n}.
\begin{lemma}\label{M Lemma}
Let $K_{q,n}=\max_j\|\xi_{j,n}\|_{q}$. Then
$$
\|R\|_{p,n}:=\max_j\|R_{j,n}\|_{q}\leq  K_{q,n}\Gamma_n(1):=a_n
$$
and so, for every fixed $n$,\, $\{D_{j,n}:1\leq j\leq k_n\}$ is a  martingale difference with respect to the filtration  $\cG_{j,n}:=\cF_n(b_j(n))=\sig\{\zeta_{1,n},\zeta_{2,n},...,\zeta_{b_j(n),n}\}$. Moreover,  
$$
\|\cD_{j,n}\|_{p}\leq C_p(A_n\beta_n+a_n)
$$
for some $C_{p}>0$ which depends only on $p$.
\end{lemma}

\begin{proof}
By the definition \eqref{var pi def} of $\varpi_{q,p}(\cdot,\cdot)$ we have 
$$
\|R\|_{p,n}\leq K_{q,n}\sum_{j\geq1}\varpi_{q,p,n}(j).
$$
Using \eqref{Var po rel} we have 
$$
\sum_{1\leq j\leq n}\varpi_{q,p,n}(j)\leq 2\sum_{1\leq j\leq n}(\al_n(j))^{1/p-1/q}=2\Gamma_n(1).
$$
Hence 
$$
\|\cD_{j,n}\|_{p}\leq 2\|R\|_{p,n}+\max_{j}\|X_{j,n}\|_p
$$
and the desired estimate on $\|\cD_{j,n}\|_p$ follows from \eqref{Then} (recalling that $q>p$). 
Finally, given that $\cD_{j,n}$ is well defined it is immediate to check that $\{D_{j,n}: 1\leq j\leq k_n\}$ is a martingale difference with respect to the filtration $\{\cG_{j,n}: 1\leq j\leq k_n\}$.
\end{proof}

Let us consider the martingale-difference array $\cD_{j,n}=\sig_n^{-1}D_{j,n}$. Then  there is a constant $C=C_{q,p}>0$ so that  
\begin{equation}\label{LP B}
\max_{j}\|\cD_{j,n}\|_{q}\leq C(A_n\beta_n+a_n)/\sig_n.
\end{equation}
Set $M_n(t)=\sum_{j=1}^{b_{j_n(t)}}\cD_{j,n}$, which for each fixed $n$ is a continuous time martingale with respect to the filtration $\cH_{t}^{(n)}=\cG_{j_n(t),n}$. 

\begin{lemma}\label{LappR}
For each $n$ we have
\begin{equation}\label{Approx3.1}
d_P(\cW_n,M_n)\leq C\sig_n^{-\frac{p-2}{p+1}}a_n^{\frac{p}{p+1}}A_n^{-\frac{2}{p+1}}.
\end{equation}
\end{lemma}
\begin{proof}

We have
$$
\sup_{t\in[0,1]}\left\|\cW_n(t)-M_n(t)\right\|_{p}\leq 2\|R\|_{p,n}\sig_n^{-1}.
$$
Applying now Lemma \ref{MaxLem} we get that
\begin{equation}\label{MartApprox}
\left\|\sup_{t\in[0,1]}\left|\cW_n(t)-M_n(t)\right|\right\|_{p}\leq C\sig_n^{2/p}A_n^{-2/p}\|R\|_{p,n}\sig_n^{-1}
\end{equation}
where we have used that the $\cW_n$ and $M_n$ has at most $k_n$ deterministic jumps and $k_n\leq c\sig_n^2/A_n^2$.
Now the second part follows by applying  Lemma \ref{M Lemma} and then Lemma \ref{ProkAssLemma}.
\end{proof}

\subsection{Functional approximation of the martingale by a Brownian motion}\label{F approx}
In view of Lemma \ref{LappR}, our goal now is to estimate $d_P(M_n(t),B(t))$. By applying \cite[Theorem 4]{Courbot} with the martingale $M_n(\cdot)$ we get that 
\begin{equation}\label{CourApp}
d_P(M_n(\cdot),B(\cdot))\leq C_{p}\left(\tilde k_n^{1/2}|\ln \tilde k_n|^{1/2}+L_{p,n}^{\frac{1}{2p}}|\ln L_{p,n}|^{3/4}\right)
\end{equation}
where
$$\big<M_n\big>_t=\sum_{j=1}^{j_n(t)}\bbE[\cD_{j,n}^2|\cG_{j-1,n}],\,\cG_{j,n}=\cF_n(b_j(n))$$ 
is the quadratic variation of $M_n$,
$$
\tilde k_n=\inf\left\{\ve>0:\, \bbP\left(\sup_{t\in[0,1]}\left|\big<M_n\big>_t-t\right|>\ve\right)<\ve\right\}.
$$
is the, so-called, Ky-Fan distance between $\big<M_n\big>$ and $\big<B\big>$ 
 and 
$$
L_{p,n}=\sum_{j=1}^{j_n(1)}\|\cD_{j,n}\|_{p}^{p}
$$
is the $p$-th Lyapunov's sum.
Observe next that for every $s\geq 1$ we have\footnote{This is obtained by considering $\ve$ close to the right hand side of \eqref{QV1} and using the Markov inequality.}
\begin{equation}\label{QV1}
\tilde k_n\leq\left\|\sup_{t\in[0,1]}\left|\big<M_n\big>_t-t\right|\right\|_{s}^{\frac{s}{s+1}},
\end{equation}
that $j_n(1)\leq k_n+1=O(\sig_n^2/A_n^2)$
and that
\begin{equation}\label{Lyp}
L_{p,n}^{\frac1{2p}}\leq C\sig_n^{-\frac{p-2}{2p}}A_n^{-\frac{1}{p}}\max_{j}\|D_{j,n}\|_{p}.
\end{equation}
Therefore, in order to prove Theorems \ref{FunCLT} and \ref{FuncCLT MC} it is enough to provide estimates on the right hand side of \eqref{QV1}, which is the purpose of the next section.

\subsection{Completing the proof: quadratic variation estimates}
Our main result in this section is the following.

\begin{proposition}\label{QVEST}
Let $l_n\leq \frac{\sig_n^2}{18 A_n}$. Then there is a constant $	C>0$ so that
\begin{equation}\label{Goal}
\left\|\sup_{t\in[0,1]}\left|\big<M_n\big>_t-t\right|\right\|_{p/2}\leq C\min(\cR_1(n,l_n),\cR_2(n,l_n))
\end{equation}
where, with  $Q_n=\sqrt{\Gamma_n(1)}A_n\beta_n$,  $a_n=K_{q,n}\Gamma_n(1)$  and
$$
\cR_n(l)=\left(\beta_n+a_ n A_n^{-1}+\beta_n^2A_n\sqrt{l\Gamma_n(1)}\right)\sig_n^{-1}+\left(\frac{l^{1-2/p}(a_n^2+A_n^2\beta_n^2)}{\sig_n^{(2-4/p)}A_n^{4/p}}+\frac{l^{1/2}(Q_n+a_n^2)}{\sig_n A_n}\right)
$$
$$
+\left(a_n^2+Q_na_nA_n^{-1}\sig_n+A_n^2(1+K_{p,n}^2)\Gamma_n(1)\right)\sig_n^{-2}
$$
we have
$$
\cR_1(n,l)=\cR_n(l)+a_n(Q_nl^{-1/2}+a_nl^{-1})A_n^{-2}
$$
and 
$$
\cR_2(n,l)=\cR_n(l)+w(n,l)r_n(p,l)+x(n,l)\sig_n^{-1}
$$
where $r_n(p,l)$ was defined in \eqref{r def},
$$
w(n,l)=\frac{a_n}{A_n^2 l}+\frac{Q_n}{A_n^2l^{1/2}}+\frac{\sqrt{\Gamma_n(1)}a_n}{\sig_n A_n l^{1/2}}+
\frac{\sqrt{\Gamma_n(1)}Q_n}{\sig_n A_n}
$$
and 
$$
x(n,l)=\left(a_nl^{-1/2}+a_n^2l^{-1/2}+Q_n\right).
$$
\end{proposition}

\subsection{Proof of the functional CLT rates based on Proposition \ref{QVEST}}

\subsection*{Proof of Theorems \ref{FunCLT} and \ref{FuncCLT MC}}\label{Sec comp}
By combining  Corollary \ref{Cor Approx}, Lemma \ref{LappR} we see that 
$$
d_P(W_n,M_n)\leq d_P(W_n,\cW_n)+d_P(\cW_n,M_n)\leq 
C\left(A_n/\sig_n\right)^{\frac{q-2}{q+1}}\beta_n^{\frac{q}{q+1}}+C\sig_n^{-\frac{p-2}{p+1}}a_n^{\frac{p}{p+1}}A_n^{-\frac{2}{p+1}}.
$$
where $a_n=K_{q,n}\beta_n\Gamma_n(1)$. Next, by \eqref{CourApp} we have
$$
d_P(M_n(\cdot),B(\cdot))\leq C_{p}\left(\tilde k_n^{1/2}|\ln \tilde k_n|^{1/2}+L_{p,n}^{\frac{1}{2p}}|\ln L_{p,n}|^{3/4}\right).
$$
Next, by combining \eqref{LP B} and \eqref{Lyp}, we have
\begin{equation}\label{Lyp1}
L_{p,n}^{\frac1{2p}}\leq C\sig_n^{-\frac{p-2}{2p}}A_n^{-\frac{1}{p}}(A_n\beta_n+a_n)/\sig_n.
\end{equation}
Hence, in order to complete the proof of Theorems \ref{FunCLT} and \ref{FuncCLT MC} it is enough to estimate $\tilde k_n$, which is obtained by
 combining \eqref{QV1} with $s=p/2$ and Proposition \ref{QVEST}. All that is left to do in order to complete the proof of  Theorems \ref{FunCLT} and \ref{FuncCLT MC} is to gather the powers of $\sig_n^{-1}$ and compare between some of the other expressions in the resulting upper bounds on $d_P(W_n, B)$ to get the desired forms \eqref{ProkEst} and \eqref{ProkEst3} (this results in the terms $C_{p}C(n;q)$ appearing in Theorems \ref{FunCLT} and \ref{FuncCLT MC}).

\subsection{Proof of Corollaries \ref{Corr1} and \ref{Corr2}}
As discussed in Section \ref{MMsec}, in the circumstances of both Corollaries \ref{Corr1} and \ref{Corr2}  we can take $\beta_n=O(1)$. Moreover, since $K_{q,n}$ is bounded in $n$ and 
$$
\Gamma_n(m)\leq\sum _{j=m}^\infty(\al(j))^{1/p-1/p}
$$
under the conditions of Corollary \ref{Corr1}  we have that $\sup_n\Gamma_n^{-1}(\frac1{8 K_{q,n}})<\infty$. Since $\sup_n K_{2,n}<\infty$ we conclude that $\sup_n A_n<\infty$. Since also $\sup_n\Gamma_n(1)\leq \sum_{j=1}^\infty(\al(j))^{1/p-1/p}<\infty$ we see that 
\begin{equation}\label{Eqq}
\sup_n C(n;q)<\infty.
\end{equation}
Now, the proof of Corollary \ref{Corr1} is completed by taking $l_n=[\sig_n]$. To prove Corollary \ref{Corr2} we take $l_n=2$, so that $r_n(p,[l_n/2])=0$ and 
$$
w_n=2\sig_n^{-2(1-2/p)}+\sqrt 2\sig_n^{-1}
$$ 
which together with \eqref{ProkEst3} and  \eqref{Eqq} completes the proof of Corollary \ref{Corr2}.

\subsubsection{Proof of Proposition \ref{QVEST}}
We will split the proof into three steps, each one will be formulated as a lemma (and it will be clear that  Proposition \ref{QVEST} follows by combining all three).
The first one is the follows result.
\begin{lemma}\label{Ap1Lem}
There is a constant $C=C_{p}>0$ so that for all $n\in\bbN$,
$$
\left\|\sup_{t\in[0,1]}\left|\big<M_n\big>_t-\sum_{j=1}^{j_n(t)}\cD_{j,n}^2\right|\right\|_{p/2}\leq C\max_j\|D_{j,n}\|_{p}\sig_n^{-1}A_n^{-1}.
$$
\end{lemma}
\begin{proof}
Let $Z_j=\cD_{j,n}^2-\bbE[\cD_{j,n}^2|\cF_{0,b_{j-1}}]$. Then $\{Z_j\}$ is a martingale difference with respect to the filtration $\{\cF_{1, b_{j}}^{(n)}\}$.
By applying Proposition \ref{Prp7} with $p/2$ instead of $p$
we see that there is a constant $C=C_{p}$ so that
$$
\left\|\sup_{t\in[0,1]}\left|\big<M_n\big>_t-\sum_{j=1}^{j_n(t)}\cD_{j,n}^2\right|\right\|_{p/2}\leq 
C\sqrt{\sum_{j=1}^{j_n(1)}\|Z_j\|_{p}^2}=O(\max_j\|D_{j,n}\|_{p}\sig_n^{-1}A_n^{-1})
$$
where we have used that $j_n(1)\leq k_n+1=O(\sig_n^2/A_n^2)$. 
\end{proof}

The second step  is the following result.
\begin{lemma}\label{Lemm2.1}
Let us fix some $N\in\bbN$ so that $N\leq c_0\sig_n^2/A_n^2$ for some $c_0>0$.
Let 
$$
R_1(n,l)=\beta_n^2A_n^2\sqrt{\Gamma_n(1)l}\sig_n^{-1}
$$ 
 and $R_2(n,l)$ be the minimum of $a_n(Q_nl^{-1/2}+a_nl^{-1})A_n^{-2}$ 
and $w(n,l)r_n(l)+x(n,l)\sig_n^{-1}$ (all terms were defined in Proposition \ref{QVEST}).
Then there is a constant $C>0$ which depends on $N$ only trough $c_0$ so that for all $l<N$,
 with $\cD_{j,n}=D_{j,n}\sig_n^{-1}$, we have
\begin{equation}\label{Claim3}
\left\|\max_{m\leq N}\Big|\sum_{j=1}^{m}(\cD_{j,n}^2-\bbE[\cD_{j,n}^2])\Big|\right\|_{p/2} 
\end{equation}
$$
\leq CR_1(n,l)+CR_2(n,l)+C\left(\frac{l^{1-2/p}(a_n^2+A_n^2\beta_n^2)}{\sig_n^{(2-4/p)}A_n^{4/p}}+\frac{l^{1/2}(Q_n+a_n^2)}{\sig_n A_n}\right).
$$
\end{lemma}

The proof of Lemma \ref{Lemm2.1} is relatively technical and long, and it is postponed to Section \ref{Pf}.

The  third step in the proof of Proposition \ref{QVEST} is the following result.

\begin{lemma}\label{LastLem}
There is a constant $C_{p}>0$ so that for all $n$ and $t\in[0,1]$,
$$
\left|\sum_{j=1}^{j_n(t)}\bbE[\cD_{j,n}^2]-t\right|\leq C_{p}\left(a_n^2+Q_na_nA_n^{-1}\sig_n+A_n^2(1+K_{p,n}^2)\Gamma_n(1)\right)\sig_n^{-2}.
$$
\end{lemma}
\begin{proof}
First, by the orthogonality property if martingale differences we have
$$
\sum_{j=1}^{j_n(t)}\bbE[\cD_{j,n}^2]=\left\|\sum_{j=1}^{j_n(t)}\cD_{j,n}\right\|_2^2.
$$
Let $X=\sum_{j=1}^{j_n(t)}\cD_{j,n}$ and $Y=\sig_n^{-1}\sum_{j=1}^{j_n(t)}X_{j,n}$. Then  $\|X-Y\|_2\leq 2\sig_n^{-1}\|R\|_{2,n}$ and so 
\begin{equation}\label{FF}
\left|\sum_{j=1}^{j_n(t)}\bbE[\cD_{j,n}^2]-\frac{\sig_{b_{j_n(t)}}^2}{\sig_n^2}\right|=\left|\bbE[X^2]-\bbE[Y^2]\right|\leq\|X-Y\|_2\|X+Y\|_{2}\leq\|X-Y\|_2\left(\|X-Y\|_2+2\|Y\|_2\right) 
\end{equation}
$$
\leq
C_{p}\|R\|_{2,n}(\|R\|_{2,n}+Q_nA_n^{-1}\sig_n)\sig_n^{-2}
$$
where we have also used Lemma \ref{Norm Con} and that $j_n(t)\leq j_n(1)\leq k_n\leq c\sig_n^2/A_n^2$.
Now the lemma follows from Lemma \ref{Lemm0}.
\end{proof}

\begin{proof}[Proof of Proposition \ref{QVEST}]
Proposition \ref{QVEST} follows by a direct combination of Lemma \ref{Ap1Lem}, Lemma \ref{Lemm2.1} and Lemma \ref{LastLem}.
\end{proof}

\subsection{Proof of Lemma \ref{Lemm2.1}}\label{Pf}

We first need the following result.
\begin{lemma}\label{Lemm1.1}
Fix some $n\in\bbN$. Let  $N$ be a positive integer so that $N\leq c_0\sig_n^2/A_n^2$ for some $c_0>0$.
Fix some $l\in\bbN$ so that $l<N$ and for $r=1,2,...,[N/l]-1$ let $$J_r=\{l(r-1)<j\leq lr\},$$ while for $r=[N/l]$ let $J_{[N/l]}$ be the relative complement of the union of $J_1,...,J_{[N/l]-1}$ in $\{1,2,...,N\}$. For each $r$ set
 $$V_r=V_{r,n}=\left(\sum_{j\in J_r}X_{j,n}\right)^2=\left(\sum_{j\in J_r}\sum_{u\in B_j(n)}\xi_{u,n}\right)^2.$$
Set also $U_{r}=V_{r}-\bbE[V_{r}]$ and $\hat U_{r,n}=U_{r}\sig_n^{-2}$.
Then there is a constant $C=C_{q,p}>0$ so that 
\begin{equation}\label{Need1}
\left\|\max_{m\leq [N/l]}\Big|\sum_{r=1}^{m}\hat U_{r,n}\Big|\right\|_{p/2}\leq Cc_0\beta_n^2A_n^2\sqrt{\Gamma_n(1)l}\sig_n^{-1}.
\end{equation}
\end{lemma}

\begin{proof}
It is enough to show that
\begin{equation}\label{Norm Con Need1}
\left\|\max_{m\leq [N/l]}\Big|\sum_{r=1}^{m}U_{r}\Big|\right\|_{p/2}\leq C\sqrt{Nl}A_n^2\beta_n^2\sqrt{\Gamma_n(1)}.
\end{equation}
To prove \eqref{Norm Con Need1}, we first note that by Lemma \ref{Norm Con} we have $\max_r\|U_r\|_{q/2}\leq CA_n^2\beta^2_nl$.  Therefore, by the H\"older inequality and the definition  \eqref{var pi def} of the mixing coefficients $\varpi_{q/2,p/2}(\cdot,\cdot )$ and \eqref{Var po rel}, for each $r$ we have
$$
\left\|U_r\sum_{s\geq r}\bbE[U_s|\cG_{r,n}]\right\|_{p/4}\leq\|U_r\|_{q/2}\sum_{s\geq r}\|\bbE[U_s|\cG_{r,n}]\|_{p/2}\leq (CA_n^2l\beta_n^{2})^2\sum_{j}\varpi_{q/2,p/2,n}(j)$$$$\leq
C(A_n^2l\beta_n^{2})^2\sum_{j}(\al_n(j))^{\frac{2}{p}-\frac{2}{q}}\leq C(A_n^2l\beta_n^{2})^2\Gamma_n(1).
$$
Now \eqref{Norm Con Need} follows from Proposition \ref{Prp7} applied with $p/2$ instead of $p$ and the random variables $U_r$.
\end{proof}

\begin{lemma}\label{EstL}
Let $n_0,N,l$ and $J_r$ be as specified in Lemma \ref{Lemm1.1}.
 Let us define $G_r=\left(\sum_{j\in J_r}D_{j,n}\right)^2$ and $H_r=G_r-\bbE[G_r]$.
\vskip0.2cm
(i) There is a constant $C=C_{p}>0$ so that
\begin{equation}\label{RHS0}
\sig_n^{-2}\left\|\max_{u\leq [N/l]}\big|\sum_{r=1}^{u}(H_r-U_r)\big|\right\|_{p/2}\leq Cc_0a_n(Q_nl^{-1/2}+a_nl^{-1})A_n^{-2}. 
\end{equation}

\vskip0.2cm
(ii)  Let $r_n(l)=\max_{j}\left\|\sum_{s\geq j+1}\bbE[\xi_{s,n}|\zeta_{1,n},...,\zeta_{j,n}]-\bbE[\xi_{s,n}|\zeta_{j-[l/2],n},...,\zeta_{j,n}]\right\|_{p}$.
 Then there is a constant $C=C_p>0$ so that
\begin{equation}\label{RHS}
\sig_n^{-2}\left\|\max_{u\leq [N/l]}\Big|\sum_{r=1}^{u}(H_r-U_r)\Big|\right\|_{p/2}\leq Cc_0\left(w(n,l)r_n(l)+x(n,l)\sig_n^{-1}\right)
\end{equation}
where
$$
w(n,l)=\frac{a_n}{A_n^2 l}+\frac{Q_n}{A_n^2l^{1/2}}+\frac{\sqrt{\Gamma_n(1)}a_n}{\sig_n A_n l^{1/2}}+
\frac{\sqrt{\Gamma_n(1)}Q_n}{\sig_n A_n}
$$
and 
$$
x(n,l)=\left(a_nl^{-1/2}+a_n^2l^{-1/2}+Q_n\right).
$$
\end{lemma}

\begin{proof}[Proof of Lemma \ref{EstL}]
(i) Let us fix some $l\in\bbN$ so that $l<N$. Let $J_r,V_r,U_r$ be as in Lemma \ref{Lemm1.1}.
First, we have 
$$
\|G_r-V_r\|_{p/2}\leq
\left\|\left(\sum_{j\in J_r}X_{j,n}-\sum_{j\in J_r}D_{j,n}\right)\left(\sum_{j\in J_r}X_{j,n}\right)\right\|_{p/2}
+$$
$$\left\|\left(\sum_{j\in J_r}X_{j,n}-\sum_{j\in J_r}D_{j,n}\right)\left(\sum_{j\in J_r}D_{j,n}\right)\right\|_{p/2}.
$$
Now, by the definition of $D_{j,n}$ and Lemma \ref{M Lemma},
$$
\left\|\sum_{j\in J_r}X_{j,n}-\sum_{j\in J_r}D_{j,n}\right\|_{p}\leq 2\|R\|_{p,n}\leq 2a_n
$$
and so by the H\"older inequality and  Lemma \ref{Norm Con},
\begin{equation}\label{Gr Vr}
\|H_r-U_r\|_{p/2}\leq 2\|G_r-V_r\|_{p/2}\leq C
(\|R\|_{p,n}(Q_nl^{1/2}+\|R\|_{p,n})\leq Ca_n(Q_nl^{1/2}+a_n):=CE_{p,n,l}.
\end{equation}
Hence, 
\begin{equation}\label{Hence}
\left\|\max_{u\leq [N/l]}\Big|\sum_{r=1}^{u}(H_r-U_r)\Big|\right\|_{p/2}\leq\sum_{r=1}^{[N/l]}\|H_r-U_r\|_{p/2}\leq C[N/l]E_{p,n,l}.
\end{equation}
\vskip0.2cm

(ii) We first note that by the definitions of $D_{j,n}$ and $R_{j,n}$ there are $\al_{r,n}<\beta_{r,n}$ so that 
$$\sum_{j\in J_r}D_{j,n}=\sum_{j\in J_r}X_{j,n}+R_{\beta_r,n}-R_{\al_r,n}.$$
Let us now define $R_{j,n}^{(l)}=\sum_{s\geq j+1}\bbE[\xi_{s,n}|\zeta_{j-[l/2],n},\zeta_{j-[l/2]+1,n},...,\zeta_{j,n}]$. Then by the definition \eqref{r def} of $r_n(l)=r_n(p,l)$, 
\begin{equation}\label{cbn}
\|R_{j,n}-R_{j,n}^{(l)}\|_{p}\leq r_n(l).
\end{equation}
Set
$A_r=A_{r,n}=R_{\beta_r,n}-R_{\al_r,n}$ and $A_{r,l}=A_{r,l,n}=R_{\beta_r,n}^{(l)}-R_{\al_r,n}^{(l)}$. Then, similarly to Lemma \ref{M Lemma} we have 
$$
\|A_{r,l}\|_{p}\leq a_n.
$$
Moreover, by \eqref{cbn},
$$
\|A_{r}-A_{r,l}\|_{p}\leq 2r_n(l).
$$
Let $G_r^{(l)}=G_{r,n}^{(l)}=\big(\sum_{j\in J_r}X_{j,n}+A_{r,l,n}\big)^2$ and $H_r^{(l)}=H_{r,n}^{(l)}=G_{r}^{(l)}=\bbE[G_{r}^{(l)}]$. Using the Markov inequality and Lemmas \ref{M Lemma} and \ref{Norm Con}  we have that 
\begin{equation}\label{Use}
\|H_r-H_r^{(l)}\|_{p/2}\leq 2\|G_r-G_r^{(l)}\|_{p/2}\leq2\left(\|A_r+A_{r,l}\|_{p}\|A_r-A_{r,l}\|_{p}+2\Big\|\sum_{j\in J_r}X_{j,n}\Big\|_{p}\|A_r-A_{r,l}\|_{p}\right)
\end{equation}
$$
\leq Cr_n(l)(a_n+Q_nl^{1/2})
$$
and so,
$$
\sum_{u=1}^{[N/l]}\|H_r-H_r^{(l)}\|_{p_0/2}\leq CNl^{-1}r_n(l)(a_n+Q_nl^{1/2}).
$$

Next, set $Z_r=Z_{r,n}=H_{r}^{(l)}-U_r$. Then $Z_r$ is a function of $\{X_{j,n}: j\in J_{r,l}\}$
where $J_{r,l}=J_r-\{0,1,...,l/2\}$. Moreover, using \eqref{Gr Vr} and \eqref{Use} we have
$$
\|Z_r\|_{q/2}\leq C\left(E_{p,n,l}+r_n(l)(a_n+Q_nl^{1/2})\right):=\mathscr X(n,l).
$$
Therefore, by first applying the H\"older inequality we see that
$$
\max_{r}\left\|Z_r\sum_{s\geq r}\bbE[Z_s|Z_1,...,Z_r]\right\|_{p/4}\leq \|Z_r\|_{p/2}\left(\|Z_{r}\|_{p/2}+\sum_{s=1}^{n}\|Z_s\|_{q/2}\varpi_{q/2,p/2,n}(sl)\right)$$
$$\leq 
C\mathscr X^2(n,l)\left(1+\Gamma_n(1)\right)\leq C' \mathscr X^2(n,l)\Gamma_n(1)
$$
where we have also used \eqref{Var po rel}.
Hence, by Proposition \ref{Prp7} we have
$$
\left\|\max_{s\leq [N/l]}\Big|\sum_{r=1}^sZ_r\Big|\right\|_{p/2}\leq C''\sqrt{Nl^{-1}\Gamma_n(1)}\mathscr X(n,l)
$$
and the second part follows from the above estimates and our assumption that $N\leq c_0\sig_n^2/A_n^2$.
\end{proof}

\begin{proof}[Proof of Lemma \ref{Lemm2.1}]
Let us write $J_r=\{\al_r,\al_r+1,...,\beta_r\}$ and for each $s\leq N$ let $u_s$ be so that $s\in J_{u_s}$.
Then, with $Z_j=Z_{j,n}=\cD_{j,n}^2-\bbE[\cD_{j,n}^2]$ we have
$$
\cM_N:=\max_{1\leq s\leq N}\left|\sum_{j=1}^{s}Z_j-\sum_{j=1}^{\beta_{u_s}}Z_j\right|\leq 
\max_{1\leq r\leq [N/l]}\max_{z\in J_r}\left|\sum_{j=z+1}^{\beta_r}Z_j\right|.
$$
Now, the total number of indexes taken in the above two maximums is $O(N)$. Therefore, by applying Lemma \ref{MaxLem} with  we get
\begin{equation}\label{MaxInq0}
\|\cM_N\|_{p/2}\leq C\big(N/l\big)^{2/p}\max_{r\leq [N/l],\, z\in J_r}\left\|\sum_{j=z+1}^{\beta_r}Z_j\right\|_{p/2}\leq Cc_0\sig_n^{-(2-4/p)}l^{1-2/p}A_n^{-4/p}(a_n^2+A_n^2\beta_n^2):=w_n
\end{equation}
where in the second inequality we have used \eqref{LP B} and that $\beta_r-z\leq\beta_r-\al_r\leq 2l$. Here $C,C'$ are some positive constants. Hence
\begin{equation}\label{MaxInq}
\left\|\max_{m\leq N}\Big|\sum_{j=1}^{m}(\cD_{j,n}^2-\bbE[\cD_{j,n}^2])\Big|\right\|_{p/2}\leq  
w_n+\left\|\max_{s\leq N}\Big|\sum_{j=1}^{s}Z_j\Big|\right\|_{p/2}. 
\end{equation}

Next, in order to estimate the second term on the right hand side of \eqref{MaxInq}, let $G_r=\left(\sum_{j\in J_r}D_{j,n}\right)^2$ be as defined in Lemma \ref{EstL}, and set $T_r=G_r-\sum_{j\in J_r}D_{u,n}^2$. Then $\{T_r\}$ is a martingale difference with respect to the filtration $\cG_r=\cF_{1,b_{rl}(n)}^{(n)}$, where $b_j(n)$ is the right end point of the block $B_j(n)$. Moreover, since 
$$
\max_{u\leq v}\left\|\sum_{j=u}^vD_{j,n}-\sum_{j=u}^v X_{j,n}\right\|_{q}\leq 2\|R\|_{q,n}\leq 2a_n,
$$
 using \eqref{MaxMom}, Lemma \ref{Norm Con} and that $|J_r|\leq 2l$  we have
\begin{equation}\label{Tr bound}
\|T_r\|_{p/2}\leq Cl\left(Q_n+a_n^2\right).
\end{equation}
Next, observe that $\sum_{j=n_0+1}^{\beta_{u_s}}Z_j=\sum_{r=1}^{u_s}\sum_{j\in J_r}Z_j$ and so, with $H_r=G_r-\bbE[G_r]$ we have
\begin{equation}\label{EE}
\sum_{j=1}^{\beta_{u_s}}Z_j=\sig_n^{-2}\sum_{r=1}^{u_s}H_r-\sig_n^{-2}\sum_{r=1}^{u_s}T_r.
\end{equation}
Now, by Proposition \ref{Prp7} applied with the martingale $\{T_r\}$,
\begin{equation}\label{EEE}
\left\|\max_{u\leq [N/l]}\Big|\sum_{r=1}^{u}T_r\Big|\right\|_{p/2} \leq C_{p}\sqrt{\sum_{r=1}^{[N/l]}\|T_r\|_{p/2}^2}\leq C_{p}'\left(Q_n+a_n^2\right)\sqrt{Nl}.
\end{equation}
where in the second inequality we have  used \eqref{Tr bound}.
Finally,  Lemma \ref{Lemm2.1} follows from applying \eqref{MaxInq}, \eqref{EE} and \eqref{EEE} and then applying Lemma \ref{Lemm1.1}, Lemma \ref{EstL} in order to estimate the term\footnote{Note that $$\left\|\max_{s\leq N}|\sum_{r=1}^{u_s}H_r|\right\|_{p/2}\leq\left\|\max_{u\leq [N/l]}|\sum_{r=1}^{u}H_r|\right\|_{p/2}.$$}
 and using that $N\leq c_0\frac{\sig_n^2}{A_n^2}$.
\end{proof}

\section{Additional results: proofs}

\subsection{A Berry-Esseen theorem via the Stein-Tikhomirov method}

\begin{proof}[Proof of Theorem \ref{BEthm}]
Fix some $n$. Let $X_{j,n}=\sum_{k\in B_{j}(n)}\xi_{k,n}$, $j=1,2,...,k_n$ be as defined in Section \ref{Block}.
Recall that in \eqref{Then} we derived that
\begin{equation}\label{Then1}
\|X_{j,n}\|_{q}\leq 4\beta_n A_n.
\end{equation}
Since $X_{j,n}$ is a function of $\{\zeta_{k,n}: j\in B_j(n)\}$ for $j<k_n$ and $X_{k_n}$ is a function of $\{\zeta_{k,n}: j>b_{k_n-1}(n)\}$ we obtain that if $A$ is measurable with respect to $\sig\{X_{1,n},...,X_{k,n}\}$ and $B$ is measurable with respect to $\sig\{X_{k+m,n},...,X_{k_n,n}\}$, where $k+m\leq k_n$, then 
$$
|\bbP(A\cap B)-\bbP(A)\bbP(B)|\leq \al_n(m).
$$
Thus, the $\al$-mixing coefficients $\tilde\al_n(j)$ of the array $\{X_{1,n},...,X_{k_n,n}\}$ do not exceed  the ones of $\{\zeta_{1,n},...,\zeta_{n,n}\}$. 
Now Theorem \ref{BEthm}(i) follows from applying \cite[Lemma A]{Sku} with $X_j=X_{j,n}$ and $s=q$ with $k=c\ln \sig_n$ and $h=R_n\sig_n^{\gamma_n}$. We note that this pair $(k,h)$ will satisfy the conditions of \cite[Lemma A]{Sku} only when $B_n=o(\sig_n)$, but when the latter fails the upper bound in Theorem \ref{BEthm}(i) is larger than one, so Theorem \ref{BEthm}(i)  trivially holds true. After plugging in these $h$ and $k$ we use that $\|X_{j,n}\|_{L^q}=O(\be_n A_n)$ to bound the term $x$ appearing in Theorem \cite[Lemma A]{Sku}, and we use that $\sum_{j=1}^{k_n}\tilde \al_n(j)\leq B_n\sum_{j=1}^{k_n}j^{-a_n}=O(B_n\sig_n^{1-a_n})$ to bound the term $\al$ appearing in \cite[Lemma A]{Sku}, which yields the desired upper bound (the exact choice of the power $\gamma_n$ comes from  comparing between some of the powers of $\sig_n^{-1}$ which appear in the upper bound from \cite[Lemma A]{Sku}).

 Theorem \ref{BEthm}(ii) follows from applying \cite[Lemma A]{Sku} with $X_j=X_{j,n}$ and $s=q$ with $k=c\ln \sig_n$ and $h=\frac{\ln R_n+C \ln \sig_n }{\ka_n}$ for $C>c$ large enough. 
We note that this pair $(k,h)$ will satisfy the conditions of \cite[Lemma A]{Sku} only when $B_n=o(\sig_n^{1/2})$ and $\ka_n^{-1}=o(\sig_n^{1/2})$, but when the latter fails the upper bound in Theorem \ref{BEthm}(i) is larger than one, so Theorem \ref{BEthm}(i) will trivially hold true. After plugging in these $h$ and $k$ we use that $\|X_{j,n}\|_{L^q}=O(\be_n A_n)$ to bound the term $x$ appearing in \cite[Lemma A]{Sku}, and we use that $\sum_{j=1}^{k_n}\tilde \al_n(j)\leq B_n\sum_{j=1}^{k_n}e^{-j\ka_n}=O(B_n\ka_n^{-1})$ to bound the term $\al$ appearing in \cite[Lemma A]{Sku}.

\end{proof}

\begin{remark}\label{BErem}
Let us suppose that $\phi(j)=O(j^{-\te})$ for some $\te>4$. Suppose also that $\beta_n$ from  \eqref{MaxMom} is defined with $q=4$ and that $Q_n=O(1)$. Set $X_{j,n}=\sig_n^{-1}X_{j,n}$.
Let us consider the graph $\cG_{n,r}=(V,\cE)=(V_{n,r},\cE_{n,r})$, where $V_n=\{1,2,...,k_n\}$ and $(j_1,j_2)\in\cE_{n,r}$ if and only if $|j_1-j_2|\leq r$. Then the size of a ball around any point in the graph is at most $2r$. Let us denote by $N_v$ the unit ball around $v\in V$ in this graph, and let $N_v^c=V\setminus N_v$. Using this ``weak-dependence" graph applying
 \cite[Theorem 1.2.2]{HK} with $D=6r$, $r=r_n=\sig_n^{3/(\te+2)}$ and $\rho=1$, and then  \cite[Lemma 1.2.3]{HK} with $p=q$ (for $p$ large enough) and  \cite[Lemma 1.2.5]{HK}, using also Proposition \ref{PropMix} we obtain that 
 $$
\sup_{t\in\bbR}|\bbP(S_n/\sig_n\leq t)-\Phi(t)|\leq A\sig_n^{-\zeta(\te)}
$$
where $\Phi$ is the standard normal distribution function, and $\zeta(\te)\to 1$ as $\te\to\infty$.
In fact, our computation shows that we can take
$\zeta(\te)=\frac{\te}{\te+1}\min(1-2\ve_\te,\ve_\te\te-2,\ve_\te(\te/2-1))$, $\ve_\te=\frac{3}{\te+2}$.
Some rate can also be obtained when $\phi_n(j)$ grows sufficiently moderately in $n$. Moreover, by applying  \cite[Theorem 1.2.1]{HK} we also obtain convergence rates in the Wasserstein metric.
\end{remark}

\subsection{Moderate deviations and moments estimates via the method of cumulants: proof of Theorems \ref{MDPthm} and \ref{RoseIneq}}

\subsection{Interlaced mixing properties}
 We will need first the following
general result. Let $U_i,\,i=1,2,...,L$ be  $d_i$-dimensional random vectors defined on the
probability space $(\Om,\cF,P)$ from Section \ref{Main}, and
$\{\cC_1,\cC_2\}$ be a partition of $\{1,2,...,L\}$. 
Consider the random vectors $U(\cC_j)=\{U_i: i\in\cC_j\}$, $j=1,2$, 
and let 
\[
U^{(j)}(\cC_i)=\{U_i^{(j)}: i\in\cC_j\},\,\, j=1,2
\]
be independent copies of the $U(\cC_j)$'s.
For each  $1\leq i\leq L$ let $a_i\in\{1,2\}$ be the unique index
such that $i\in\cC_{a_i}$.

In the course of the proof of Theorems \ref{FunCLT}, \ref{BEthm} and \ref{MDPthm} we will need the following result which appears as  Corollary 1.3.11 in \cite{HK}. 

\begin{lemma}\label{lem3.1-StPaper}
Let $\cH_{k,l}$ be a nested\footnote{That is $\cH_{k,l}\subset\cH_{k',l'}$ if $[k,\ell]\subset[k',\ell']$} family of sub-$\sig$-algebras in some probability space.
Suppose that each $U_i$ is $\cH_{m_i,n_i}$-measurable, where $n_{i-1}<m_i\leq n_i<m_{i+1}$, 
$i=1,...,L$, $n_0=0$ and $m_{L+1}=k_n+1$.   Let $\phi(\cdot)$ be the $\phi$-mixing coefficients corresponding to the family $\cH=\{\cH_{k,l}\}$, namely 
$$
\phi(k)=\sup\{\phi(\cH_{u,v},\cH_{u_1,v_2}): u_1-u\geq k\}.
$$
Then,
\begin{equation}\label{alpha cor}
\al\big(\sig\{U(\cC_1)\},\sig\{U(\cC_2)\}\big)\leq4\sum_{i=2}^L\phi(m_i-n_{i-1})
\end{equation}
where $\sig\{X\}$ stands for the $\sig$-algebra generated by a random variable 
$X$ and $\alpha(\cdot,\cdot)$ are the $\al$ mixing coefficients.
\end{lemma}

Next, fix $n\in\bbN$.
\begin{proposition}\label{PropMix}
Let $X_{j,n}$ be the blocks constructed in Section \ref{Block}.
Let $\Del_1,\Del_2$ be two subsets of $\{1,2,...,k_n\}$ be so that $\text{dist}(\Del_1,\Del_2)\geq b$ so some $b\in\bbN$. Let $\Del=\Del_1\cup\Del_2$ and suppose that $\Del$ is a disjoint union of $L$ sets which are contained either in $\Del_1$ or in $\Del_2$.
 Then
\begin{equation}\label{alpha cor''}
\al\big(\sig\{X_{i,n}: i\in \Del_1\}, \sig\{X_{j,n}: j\in \Del_2\}\big)\leq 4L\phi_n(b).
\end{equation}
\end{proposition}

\begin{proof}
Let us write
\[
\Del:=\Del_1\cup \Del_2=\bigcup_{i=1}^{L}C_i
\]  
where $c_i+b\leq c_{i+1}$ for any $c_i\in C_i$ and $c_{i+1}\in C_{i+1}$,\, $i=1,2,...,L-1$
and each one of the $C_i$'s is either a subset of $\Del_1$ or a subset of $\Del_2$. Let us define $V_j=\{X_{u,n}: u\in C_j\}$. Then $V_j$ is  a function of the vector $U_j$ whose components are the $\xi_{k,n}$'s which appear as summands in one of $X_{u,n}$, $u\in C_j$. Hence, there are numbers $a'_j$ so that  $a'_j+b\leq a'_{j+1}$ and $U_j=\{\xi_k:\,a'_j\leq k\leq a'_{j+1}-b\}$. Applying now Lemma \ref{lem3.1-StPaper} with the partition $\{\cC_1,\cC_2\}$, where $\cC_i$ is the set of indexes $j$ so that $C_j\subset\Del_i$, $i=1,2$ we  obtain \eqref{alpha cor''}.
\end{proof}

\subsection{A moderate deviations principle and Rosenthal type inequalities} 
Let $S$ be a random variable with finite moments of all orders. We recall that the
 the $k$-th cumulant of $S$ is given by
\[
\Gam_k(S)=\frac1{i^k}\frac{d^k}{dt^k}\big(\ln\bbE[e^{itS}]\big)\big|_{t=0}.
\]
Note that $\Gamma_k(aS)=a^k\Gamma_k(S)$ for any $a\in\bbR$.
Moreover,
 $\Gamma_1(S)=\bbE[S]$, $\Gamma_2(S)=\text{Var}(S)$.
The main ingredient in the proof of Theorem \ref{MDPthm} is the following result.
\begin{proposition}\label{Prp}
There is a constant $C>0$ so that for any $k\geq3$ we have
$$
|\Gam_k(S_n/\sig_n)|\leq (CR_n)^{k}(k!)^{1+\frac 1\eta}\sig_n^{-(k-2)}.
$$
where $R_n$ is defined in Theorem \ref{MDPthm}. 
\end{proposition}
We note that Proposition \ref{Prp} implies that  
$$
|\Gam_k(S_n/\sig_n)|\leq (CR_n^3)^{k-2}(k!)^{1+\frac 1\eta}\sig_n^{-(k-2)}.
$$
Using these estimates on the cumulants, Theorem \ref{MDPthm} follows from \cite[Theorem 1.1]{Dor}.
Moreover, the proof of Theorem \ref{RoseIneq} proceeds similarly to the proof of  \cite[Theorem 3]{Douk2} or \cite[Theorem 6.3]{Ha}.

Before we prove Proposition \ref{Prp} we need the following lemma, which is a consequence of Proposition \ref{PropMix}.
\begin{lemma}
Let $\Del_1,\Del_2\subset\bbN$ be two finite sets so that $\text{dist}(\Del_1,\Del_2)\geq b$. Let $d_1=|\Del_1|+|\Del_2|$ be the sum of their cardinalities. 
Then
\begin{equation}\label{alpha cor'}
\al\big(\sig\{X_{i,n}: i\in \Del_1\}, \sig\{X_{j,n}: j\in \Del_2\}\big)\leq 4 d_1\phi_n(b).
\end{equation}
\end{lemma}

\begin{proof}[Proof of Proposition \ref{Prp}]
Recall that $S_n=\sum_{j=1}^{k_n}X_{j,n}$.
Let us fix some $n$ and set $V=\{1,2,...,k_n\}$ and $X_j=X_{j,n}$.
Set 
$$
M_n=\max(\tilde\beta_nA_n,K_{\infty,n}j_n).
$$
where
$$
\tilde\beta_n=C_{1/6}\left(1+j_n K_{\infty,n}\right).
$$ 
Here $j_n$ is as defined in Theorem \ref{MDPthm}.
Then, by \eqref{Pel}, we see that there are constants $C_0,A_0>0$ so that for every $j$ and $p>2$, 
\begin{equation}\label{Then11}
\|X_{j,n}\|_{p}^p\leq  C_0p^pM_n^p\leq A_0^pM_n^pp!
\end{equation}
where we have used that $p^p\leq Ae^pp!$ for some $A>1$ (as a consequence of Stirling's approximation).
We will show soon that for all $k\geq1$, $k\geq 1$, $b>0$
and a finite collection $A_j,\,j\in\cJ$ of (nonempty) subsets of $V$ so that 
$\min_{i\not=j}\rho(A_i,A_j)\geq b$ and $r:=\sum_{j\in\cJ}|A_j|\leq k$ we have
\begin{equation}\label{MixGorc}
\left|\bbE\left[\prod_{j\in\cJ}\prod_{i\in A_j}X_{i,n}\right]-\prod_{j\in\cJ} \bbE\left[\prod_{j\in A_j}X_{i,n}\right]\right|\leq
C(r-1)\left(\prod_{j\in\cJ}\prod_{i\in A_j}\|X_{i,n}\|_{2k}\right)k\sqrt{\phi_n(b)}
\end{equation}
where $C$ is some constant. Once \eqref{MixGorc} is obtained
the proof of the proposition is completed  by applying (3.6) from \cite[Corollary 3.2]{Ha},  with the function $\rho(x,y)=|x-y|$, $\del=1$
 and  $\gamma_{\del}(b,k)=Ck\sqrt{\phi_n(b)}$ (the condition \eqref{MixGorc} is on of the main assumptions of \cite[Corollary 3.2]{Ha}). We note that the constant $c=c(c_0,a,u_0,\eta)$ from \cite[equation (3.3)]{Ha} was not given explicitly in  \cite[Corollary 3.2]{Ha}, however it can easily be seen from the arguments in the proof of \cite[Corollary 3.2]{Ha} that we can choose $c_0$ of the form 
$$
c_0=C(c_0,u,\eta)e^{a^{-u}}
$$
for an arbitrary $u>1$. Indeed, what is needed is to control the constants $H,\psi,c_1$ and $\psi_0$ that appear in the course of the proof of  \cite[Corollary 3.2]{Ha}. To conrol $H$, for instance, it is enough to show that $k!4^ke^{-ak^2/4}\leq Ce^{-a^{-u}}$ (because of the assumption that $m^\eta\geq k$ at the begining of the proof of \cite[Corollary 3.2]{Ha}), and this can be done by taking the logarithms of both sides, using that $\ln (k!)\leq k\ln k$ and distingusihing between two cases: $ak\geq C\ln k$ and $ak<C\ln k$ for some $C$ large enough. The desired upper bounds $O(e^{-a^{-u}})$ on  $\psi,c_1$ and $\psi_0$ can be obtained similarly.

In order to prove \eqref{MixGorc}, 
we first recall that (see Corollary A.2 in \cite{HallHyde}) for any two sub-$\sig$-algebras $\cG,\cH\subset\cF$ we have
\begin{equation}\label{alpha cov HallHyde}
\text{Cov}(\eta_1,\eta_2)\leq 8\|\eta_1\|_u\|\eta_2\|_v\big(\al(\cG,\cH)\big)^{1-\frac1u-\frac1v},
\end{equation}
where $h_1$ is  $\cG$-measurable, $h_2$ is $\cH$-measurable
and $1<u,v\leq\infty$ satisfy that $\frac1u+\frac1v<1$ (where we set $\frac1\infty=0$). 
Let us now write $\cJ=\{1,2,...,J\}$ and set $\Del_1=A_1$ and 
$\Del_2=\bigcup_{1<i\leq J}A_i$. Applying (\ref{alpha cov HallHyde}) with 
$u=\frac{2k}{|\Del_1|}$ and $v=\frac{2k}{|\Del_2|}$
we obtain that 
\begin{equation}\label{Temp prod  est}
\Big|\text{Cov}\big(\prod_{i\in \Del_1}X_{i,n},\prod_{i\in\Del_2}X_{i,n}\big)\Big|\leq
8\big\|\prod_{i\in \Del_1}X_{i,n}\big\|_u\,\big\|\prod_{i\in\Del_2}X_{i,n}\big\|_v\al^{1-\frac1{2}}
\end{equation}
where $\al=\al\big(\sig\{X_{i,n}: i\in \Del_1\}, \sig\{X_{j,n}: j\in \Del_2\}\big)$
and we have also used that $\al\leq1$ and that
\[
\frac1u+\frac 1v=\frac{|\Del_1\cup\Del_2|}{2k}=\frac{r}{2k}\leq\frac 1{2}.
\] 
Next, by \eqref{alpha cor'} we have
$$
\al=\al\big(\sig\{X_{i,n}: i\in \Del_1\}, \sig\{X_{j,n}: j\in \Del_2\}\big)\leq 4r\phi_n(b).
$$
Using the H\"older inequality to estimate the norms on the right hand side of (\ref{Temp prod  est}) 
and then repeating the above arguments with $\cJ_i=\{i,i+1,...,J\},\,i=2,3,...,J$ in place of $\cJ$ we
obtain \eqref{MixGorc}, taking into account that $J=|\cJ|\leq \sum_{i\in\cJ}|A_i|=r$, and the proof of the proposition is complete.
\end{proof}

\section{Application to local functionals of uniformly mixing arrays and sequential dynamical systems}\label{Local}
Let $Y_j$ be a sequence of  variables whose mixing coefficients are denoted by $\al(\cdot), \phi(\cdot)$ etc. 
Let $m_n\to\infty$ be  a sequence and let us consider random variables of the form
$$
\zeta_{j,n}=(Y_{j-m_n},...,Y_{j+m_m}).
$$
Then $\al_n(k)\leq \al_n(k-2m_n)$ and $\phi_n(k)\leq \phi(k-2m_n)$ for all $k>m_n$.  
Next, let $\{\xi_{j,n}: \,1\leq j\leq n\}$ be a triangular array of the form
$$
\xi_{j,n}=g_{j,n}(\zeta_{j,n})=g_{j,n}(Y_{j-m_n},...,Y_{j+m_n}).
$$
\begin{lemma}
Suppose that $\phi(n_0)<\frac12$ for some $n_0\in\bbN$.
\vskip0.1cm

(i) If $\xi_{j,n}$ are bounded then  $\beta_n$ defined through \eqref{MaxMom} satisfies
$$
\beta_n=O\left(1+m_nK_{\infty,n}\right),\,\,K_{\infty,n}=\max_j\|\xi_{j,n}\|_\infty.
$$
\vskip0.1cm
(ii) If $\sig_n\geq Cn^{\del}$ for some $\del>0$ then for every $q_0\geq q$,\, $\beta_n$ defined through \eqref{MaxMom} satisfies 
$$\beta_n=O\left(1+m_n n^{1/q_0}K_{q_0,n}\right)$$ where $K_{q_0,n}=\max_j\|\xi_{j,n}\|_{q_0}$.
\end{lemma}
\begin{proof}
The lemma follows from the discussion in Section \ref{MMsec} and noting that we can take $j_n=2m_n+n_0=O(m_n)$.
\end{proof}

\begin{remark}
We refer to \cite{FZ} for a class of examples which includes covariance estimators. In this setup $\xi_{j,n}$ are uniformly bounded in $L^q$ for some $q>2$ and $\lim_{n\to\infty}\frac 1n\text{Var}(S_n)=\sig^2>0$, which means that for these applications it is unnecessary to use $\beta_n$ from \eqref{MaxMom} is  (as discussed in Remarks \ref{R beta} and \ref{R beta 1}).
\end{remark}

Next, we have the following. 

\begin{lemma}
Suppose that $\sum_{j\geq 0}\left(\al(j)\right)^{1/p-1/q}<\infty$. Let $u_n$ be the first index so that 
$$
\sum_{j\geq u_n}\left(\al(j)\right)^{1/p-1/q}\leq \frac{1}{8K_{q,n}}.
$$
Then the number $A_n$ defined in \eqref{A_n def} satisfies
$$
A_n\leq c_0\max(u_n+2m_n,K_{2,n})
$$
for some constant $c_0$. In particular, if $\sup_{j,n}\|\xi_{j,n}\|_{q}<\infty$ then $A_n=O(m_n)$.
\end{lemma}
\begin{proof}
The lemma follows since when $m\geq 2m_n+u$ we have 
$$
\Gamma_n(m)\leq\sum_{j\geq u}\left(\al(j)\right)^{1/p-1/q}.
$$
\end{proof}
Once we have estimates on $A_n$ and $\beta_n$ we can apply all the results in Section \ref{Main}. Note that in order for the results to be meaningful we need to assume that $m_n=O(\sig_n^{t_0})$ for a sufficiently small $t_0$. Thus, in order to get explicit restrictions on the ``allowed" growth rate of $m_n$, we need to have a certain \textit{a-priori} lower bounds on $\sig_n$ for $n$ large enough, which is the case in the examples in \cite{FZ} (for which $\sig_n^2/n\to\sig^2>0$).
\begin{remark}\label{RR}
When $Y_n$ is a Markov chain then $\{\xi_{1,n},...,\xi_{j,n}\}$ is a Markov chain with memory $2m_n$. Thus, if $m_n=o(\sig_n^2)$ then we can apply
 Theorem \ref{FuncCLT MC} with $l_n=4m_n$ for which $r_n(p, [l_n/2])=0$. We thus see that if  $m_n=o(\sig_n^{a})$ for some $a<1/5$, $\sup_{j,n}\|\xi_{j,n}\|_{q}<\infty$ and  $\sum_{j\geq 0}\left(\al(j)\right)^{1/p-1/q}<\infty$ then 
 $$
 d_P(W_n, B)=O(\sig_n^{-w(p)+5a/2}), w(p)\to 1/2 \text{ when } p\to\infty.
 $$
When $\sig_n^2$ grows linearly fast in $n$ then we can replace $A_n$ by $1$ and $\beta_n$ by $K_{q,n}$ and get that if $a<1$ then 
 $$
 d_P(W_n, B)=O(\sig_n^{-w(p)+a/2}), w(p)\to 1/2 \text{ when } p\to\infty.
 $$ 
In particular we get the above rates in the setup of \cite{FZ}.
\end{remark}

\subsection{A linear example with non-linearly growing variances}\label{Sec7.2}
Let us conclude this section with a local example of a  triangular array whose variance grows essentially at an arbitrary rate.
We assume here that $Y_j$ are real-valued and that $\bbE[Y_j]=0$, $\bbE[Y_j^2]=1$ and $\bbE[Y_jY_k]=0$ for $k\not=j$. Moreover, we assume that $Y_j$ are $\phi$-mixing and bounded in $L^q$ for some $q>2$ (for instance, $Y_j$ can be independent).
Let $a_j$ be a positive deceasing bounded sequence and let 
$$
\xi_{k,n}=\sum_{j=k}^{k+m_n}a_jY_{j}.
$$
Then $\|\xi_{k,n}\|_{q}\leq C\sum_{j=0}^{m_n}a_j$. Let us take now $a_j=j^{-\gamma}$ for some $\gamma<1/2$. Then $\|\xi_{k,n}\|_q\leq C'm_n^{1-\gamma}$.

\begin{lemma}
Suppose that $m_n=o(n)$.
There are constants $C_1,C_2>0$ so that for all $n$ large enough we have 
$$
C_1m_n^2 n^{1-\gamma}\leq \sig_n^2\leq C_2 m_n^2 n^{1-\gamma}.
$$
\end{lemma}
\begin{proof}
Let $k<s$. If $s-k>2m_n$ then $\bbE[\xi_{k,n}\xi_{s,n}]=0$. Otherwise,
$$
\bbE[\xi_{k,n}\xi_{s,n}]=\sum_{|j_1|\leq m_n}\sum_{|j_2|\leq m_n}a_{k+j_1}a_{s+j_2}\bbI(k+j_1=s+j_2)=\sum_{r=s-m_n}^{k+m_n}a_{r}^2
$$
and so 
$$
\sig_n^2=\sum_{k=1}^n\bbE[\xi_{k,n}^2]+2\sum_{k=1}^n\sum_{k<s<k+2m_n}\bbE[\xi_{k,n}\xi_{s,n}]=
\sum_{k=1}^n\sum_{j=0}^{2m_n}(2j+1)a_{k-m_n+2j}.
$$
Now, notice that for all $k>5m_n$ and all $0\leq j\leq 2m_n$ we have 
$$
(2j+1)a_{k+m-2j}=(2j+1)(k+m-2j)^{-\gamma}\geq C(2(2m_n)+1)a_{k+3m_n}
$$
where $C$ is a positive constant.
Hence 
$$
 \sum_{k=1}^n\sum_{j=0}^{2m_n}(2j+1)a_{k-m_n+2j}^2\geq  \sum_{k=5m_n}^n\sum_{j=0}^{2m_n}(2j+1)a_{k-m_n+2j}^2\geq 
C(2m_n)(4m_n+1)\sum_{k=5m_n}^{n}a_{k+3m_n}\geq C'm_n^2n^{1-\gamma}.
$$
On the other hand,
 since $a_j=j^{-\gamma}$ is decreasing and positive we have
$$
\sum_{k=1}^n\sum_{j=0}^{2m_n}(2j+1)a_{k-m_n+2j}^2\leq Cm_n^2+\sum_{k=2m_n}^n(2m_n)(4m_n+1)
a_{k-m_n}^2 $$$$\leq Cm_n^2+C_\gamma \sum_{k=m_n}^{n-m_n}(2m_n)(4m_n+1)a_{k}^2\leq Cm_n^2 n^{1-\gamma}
$$
where we have used that $a_{k-m_n}\leq C_\gamma a_k$ if $k>2m_n$. 
\end{proof}
Let us now show that all results in this paper are effective when $m_n=[n^{\del_0}]$ for $\del_0$ small enough, that is, let us provide some effective estimates on $\beta_n$.
Suppose now that $K_{q,n}=O(n^{(1-\gamma)\del_0})$.  Then, as discussed in Section \ref{MMsec}, we have 
$$
\beta_n=O(n^{(1-\gamma)\del_0+1/q+\del_0}).
$$
On the other hand, $\sig_n\geq Cn^{(1-\gamma)/2+\del_0}$. Thus,  for every $r_0>0$ if $q$ is large enough and $\del_0$ is small enough we have $\beta_n=O(\sig_n^{r_0})$. 

\begin{remark}
We can also consider the case when $Y_j=Y_j^{(n)}$ depends on $n$, and in this case instead of assuming that $\bbE[Y_i^{(n)}Y_j^{(n)}]=0$ for $i<j$ we can assume that $|\bbE[Y_i^{(n)}Y_j^{(n)}]|\leq b_n$ for some sequence $b_n$ which decays sufficiently fast to $0$ as $n\to\infty$. Under such conditions we get the same growth rate for the variances.
\end{remark}

\subsection{A sketch of an application to sequential expanding dynamical systems}\label{SDS}
We consider here a class of random variable satisfying certain approximation conditions which arise naturally for certain classes of expanding maps, and a sequence of functionals so that $\sig_n$ grows at least logarithmically fast in $n$.
 For the sake of simplicity, we will focus  on random variables which can be approximated exponentially fast by mixing sequences, but the same idea works with polynomially fast approximation. Moreover, we will only consider the case of approximation by  $\phi$-mixing processes.

Let us now give a precise description. We assume here that for each $r\in\bbN$ there are random variables $\xi_{j,r}$ so that for all $p$,
\begin{equation}\label{Approx}
\sup_j\|\xi_{j}-\xi_{j,r}\|_{p}:=\beta_{p}(r)\leq c_pa^r,\,a\in(0,1)
\end{equation}
and that $\{\xi_{2rj,r}\}$ decay as $O(j^{-c})$, uniformly in $r$ for some $c>1$. Then the $\phi$-mixing coefficients $\phi(j;r)$ of $\{\xi_{j,r}\}$ satisfy
$$
\phi(j+r;r)\leq Cj^{-c}.
$$
For the sake of simplicity, let us also assume that the variables $\xi_{j}$ and $\xi_{j,r}$ are uniformly bounded.

\begin{example}
Let $(\cY_j,\cB_j)$ be measurable spaces, and
let  $T_j:Y_j\to Y_{j+1}$ be a sequence of maps. Let $\zeta_0$ be a $\cY_0$-valued random variable and let us define a $\cY_j$-valued random variable by $\zeta_j=T_{j-1}\circ\dots\circ T_1\circ T_0(\zeta_0)$. Let $g_j:Y_j\to\bbR$ be a uniformly bounded sequence of functions and set 
$\xi_j=g_j(\zeta_j)$. Then condition \eqref{Approx} holds true when $g_j$ are uniformly bounded H\"older continuous functions and $\{T_j\}$ is a sequential subshift of finite type, or when $\{T_j\}$ are uniformly piecewise expanding maps on the unit interval. In both cases $\xi_{j,r}$ is the conditional expectation of $\xi_{j}$ on the pullback by $T_{j-1}\circ\dots\circ T_1\circ T_0$ of $\cM_r=(T_0^{r})^{-1}\cM$
where $T_0^r=T_{r-1}\circ\dots\circ T_1\circ T_0$ and  $\cM$ is either the partition into cylinders of length one (in the subshift case) or the partition into monotonicity intervals.
\end{example}

\subsubsection{Limit theorems}
Set $S_{n,r}=\sum_{j=1}^{n}\xi_{j,r}$.
Then
$$
\|\max_{k\leq n}|S_k-S_{k,r}|\|_{p}\leq \sum_{j=1}^n\|\xi_{j}-\xi_{j,n}\|_p\leq n\beta_{p}(r)\leq nc_pa^r,
$$
and so for every $s>0$ there is a constant $C=C_s>0$ so that 
 when $r=\mathfrak{r}_n=C_q\ln n$  
we have 
\begin{equation}\label{S max est}
 \|\max_{k\leq n}|S_k-S_{k,r}|\|_{p}=o(n^{-s}).
 \end{equation}
If $p\geq 2$ we conclude that $\sig_{n,\mathfrak{r}_n}:=\sqrt{\text{Var}(S_{n,\mathfrak r_n})}=\sig_n+o(1)$ and hence the latter two variances are of the same magnitude. 

 Moreover, we assume that $c'=c(1/p-1/q)>1$. The following two result follow now directly from the definition \eqref{A_n def} of $A_n$ and the discussion at the beginning of Section \ref{Local}. 

\begin{lemma}\label{ll}
The sequences from Section \ref{Main} satisfy
$$
j_n=O(\ln n), \beta_n=O(\ln n),  A_n=O((\ln n)^{s}), C(n;q)=O((\ln n)^{1+s})
$$
where $s=s_{c'}$ depends only on $c'$.
\end{lemma}

\begin{lemma}\label{LLL}
We have
$$|\sig_n^{-1}-\sig_{n,\mathfrak{r}_n}^{-1}|\leq C_q\frac{\|S_n-S_{n,\mathfrak r_n}\|_{2}}{\sig_n^2}
=O(n^{-q}\sig_n^{-2}).$$
\end{lemma}
In the following sections we will explain the main ideas in the derivation of  the functional CLT, convergence rates in the CLT and moderate deviations principles for $S_n$ from the corresponding results for $S_{n,\mathfrak r_n}$ which follow from applying the theorems in Section \ref{Main} (applying the results from Section \ref{Main} with $S_{n,\mathfrak r_n}$ is done similarly to the beginning of Section \ref{Local}).   In order not to overload the paper we will not formulate explicitly the results.

\subsubsection{A Functional CLT}
Let us denote 
$$
\tilde v_n(t)=\min\{1\leq k\leq n: \text{Var}(S_{k,\mathfrak r_n})\geq t\sig_{n,\mathfrak r_n}^2\}
$$
and set 
$$
\tilde W_n(t)=\sig_{n,\mathfrak r_n}^{-1}\sum_{j=1}^{\tilde v_n(t)}(\xi_{j,\mathfrak r_n}-\bbE[\xi_{j,\mathfrak r_n}]).
$$
Let us assume that 
$\sig_n\geq c_0\ln^{s_0} n$ for some $c_0>0$ and $s_0>1+s$ and all $n$ large enough. Then applying Theorem \ref{FunCLT} we see that $\tilde W_n$ converges in distribution in the Skorokhod space $D[0,1]$ to a standard Brownian motion.

In order to prove the functional CLT for 
$$
W_n(t)=\sig_n^{-1}\sum_{j=1}^{v_n(t)}(\xi_{j}-\bbE[\xi_{j}])
$$
from the one for corresponding CLT for $\tilde W_n(\cdot)$ we need the following result.
\begin{lemma}
There is a sequence of positive numbers $\ve_n$ which converges to $0$ so that 
$$
\tilde v_n(t-\ve_n)\leq v_n(t)\leq \tilde v_n(t+\ve_n).
$$
\end{lemma}
This lemma holds true since 
$$\max_{1\leq k\leq n}|\text{Var}(S_{k})-\text{Var}(S_{k,\mathfrak r_n})|=o(1).$$ Using this lemma together with  \eqref{S max est} we conclude that the Prokhorov distance between $W_n(t)$ and $\tilde W_n(t)$ is $o(1)$.

\subsubsection{A Berry-Esseen theorem}
Applying \cite[Lemma 3.3]{BE} we get that 
$$
\sup_{t\in\bbR}|\bbP(\bar S_n/\sig_n\leq t)-\Phi(t)|\leq \sup_{t\in\bbR}|\bbP(\bar S_{n,\mathfrak r_n}/\sig_{n,\mathfrak r_n}\leq t)-\Phi(t)|+C\|\bar S_n/\sig_n-S_{n,\mathfrak r_n}/\sig_{n,\mathfrak r_n}\|_p^{\frac{p}{p+1}}
$$
where $\bar X=X-\bbE[X]$ for every random variable $X$. Using Lemma \ref{LLL}, \eqref{S max est} and some elementary estimates we get that the second term on the above right hand side is at most of order $n^{-1/2}=O(\sig_n^{-1})$, and so it is negligible in terms of the Berry-Esseen theorem, since $O(\sig_n^{-1})$ is the optimal rate. Taking into account Lemma \ref{ll}, we see that by Applying Theorem \ref{BEthm} (ii) with $S_{n,\mathfrak r_n}$ we obtain essentially the same rates $O(\sig_n^{-1/2}\ln^2\sig_n)$ for $S_n/\sig_n$ if $\sig_n\geq c_0n^{\ve}$ for some $c_0,\ve>0$. In fact, we obtain rates of the form  $O(\sig_n^{-1/2+s_c}\ln^2\sig_n)$ when  $\sig_n\geq c_0\ln^{c} n$, with $s_c$ converging to $0$ as $c\to\infty$. 

\subsubsection{A moderate deviations principle}
In general it is less clear to us how to directly obtain the moderate deviations principle for $S_n$ from the one for $S_{n,\mathfrak r_n}$ for general sets $\Gamma$. However, in certain applications we can take $p=\infty$ in \eqref{Approx} then for every diverging sequence $c_n$, a Borel set $\Gamma$, $\ve>0$ and all 
$n$ large enough
$$
\bbP(S_{n}/c_n\in\Gamma)\leq \bbP(S_{n,\mathfrak r_n}/c_n\in\Gamma^{\ve})\,\,\text{ and }\bbP(S_{n,\mathfrak r_n}/c_n\in\Gamma)\leq \bbP(S_{n}/c_n\in \Gamma^{\ve})
$$
where $\Gamma^{\ve}$ is the $\ve$-neighborhood of $\ve$. Taking $n\to\infty$ and then $\ve\to0$ we can derive the moderate deviations for $S_n$ from the corresponding one for $S_{n,\mathfrak r_n}$. Note that the variances of $S_{n,\mathfrak r_n}$ and $S_n$ are of the same order, and so it does not matter if we divide $S_n$ by $\sig_n$ or by $\sig_{n,\mathfrak r_n}$ since the ratio between these two variances can be absorbed in the sequence $(t_n)$ from Theorem \ref{MDPthm}.

\section{Appendix:Mixing coefficients and operator norms}
Let $(\Om,\cF,\bbP)$ be a probability space.
Let $\cG$ and $\cH$ be two sub-$\sig$-alegbras of $\cF$.
Recall first the definition of the following dependence coefficients between $\cG$ and $\cH$:
\begin{equation}\label{al def1}
\alpha(\cG,\cH)=\sup\left\{|\bbP(A\cap B)-\bbP(A)\bbP(B)|: A\in\cG, B\in\cH\right\},
\end{equation}
$$
\phi(\cG,\cH)=\sup\left\{|\bbP(B|A)-\bbP(B)|: A\in\cG, B\in\cH, \bbP(A)>0\right\},
$$
$$
\psi(\cG,\cH)=\sup\left\{\frac{|\bbP(B\cap A)-\bbP(B)\bbP(A)|}{\bbP(A)\bbP(B)}: A\in\cG, B\in\cH, \bbP(A),\bbP(B)>0\right\},
$$
and
\[
\rho(\cG,\cH)=\sup\left\{|\bbE[fg]|: f\in B_{0,2}(\cG), g\in B_{0,2}(\cH)\right\}
\] 
where for any sub-$\sig$-algebra $\cG$, \,$B_{0,2}(\cG)$ denotes the space of all square integrable functions $g$ so that $\bbE[g]=0$ and $\bbE[g^2]=1$.

Next, for any two sub-$\sig$-algebras $\cG$ and $\cH$ of $\cF$, and for every $p,q\geq1$ we set
\begin{equation}\label{var pi def}
\varpi_{q,p}(\cG,\cH)=\sup\left\{\left\|\bbE[h|\cG]-\bbE[h]\right\|_p: h\in L^q(\Om,\cH,\bbP), \|h\|_q=1\right\}.
\end{equation}
Then (see \cite[Ch. 4]{Br}), the classical weak dependence (mixing) coefficients can be expresses as
\begin{equation}\label{RRel}
\phi(\cG,\cH)=\frac12\varpi_{\infty,\infty}(\cG,\cH),\, \rho(\cG,\cH)=\varpi_{2,2}(\cG,\cH),\,
\al(\cG,\cH)=\frac14\varpi_{\infty,1}(\cG,\cH), \psi(\cG,\cH)=\varpi_{\infty,\infty}(\cG,\cH).
\end{equation}
By applying the Riesz–Thorin interpolation theorem \cite[Ch.6]{Folland}, if 
$$
\frac{1}p=\frac{1-\la}{p_0}+\frac{\la}{p_1}\,\text{ and }\,\,\frac{1}{q}=\frac{1-\la}{q_0}+\frac{\la}{q_1}
$$
for some $\la\in(0,1)$ and $p_0,p_1,q_0,q_1\geq1$ then
\begin{equation}\label{RT}
 \varpi_{q,p}(\cG,\cH)\leq \left(\varpi_{q_0,p_0}(\cG,\cH)\right)^{1-\la}\left(\varpi_{q_1,p_1}(\cG,\cH)\right)^{\la}.
\end{equation}
In particular, by taking $q_1=\infty$, $p_1=1$, $\la=\frac1p-\frac 1q$, $p_0=q_0=q+1-q/p$ and using the trivial upper bound $\varpi_{q_0,p_0}\leq 2$ for $q_0\geq p_0$ we see that for every $q>p>2$ and all sub-$\sig$-algebras $\cG$ and $\cH$,
\begin{equation}\label{MixCoe}
\varpi_{q,p}(\cG,\cH)\leq 2^{1+\frac 1q-\frac 1p}\left(\al(\cG,\cH)\right)^{\frac{1}p-\frac 1q}.
\end{equation}

\end{document}